\documentclass[a4paper, 11pt]{article}

\usepackage{amsmath, amscd, amsthm, amssymb, mathrsfs} 
\usepackage[applemac]{inputenc}
\usepackage[all]{xy} 
\usepackage{hyperref}
\usepackage{float}
\usepackage{verbatim}
\usepackage[colorinlistoftodos]{todonotes}
\usepackage{tocbibind}
\usepackage{mathptmx}
\usepackage{geometry}

\DeclareMathOperator{\End}{End}
\DeclareMathOperator{\Aut}{Aut}
\DeclareMathOperator{\Hom}{Hom}

\DeclareMathOperator{\proj}{proj}

\DeclareMathOperator{\vol}{vol}

\DeclareMathOperator{\AEMO}{AEMO}

\newcommand{\dbar}{\overline{\partial}}
\newcommand{\fhb}{\mathcal{O}_{\mathbb{P}(E)}(1)}
\newcommand{\PE}{\mathbb{P}(E)}

\newcommand{\LichOp}{\mathfrak{D}^* \mathfrak{D}}

\newcommand{\Lconn}{\nabla^{L^*}}
\newcommand{\Lcurv}{F^{\nabla^{L^*}}}
\newcommand{\HorT}{\mathcal{H}}
\newcommand{\VerT}{\mathcal{V}}

\newcommand{\Vtrace}{\Lambda_{\omega_{FS}}}
\newcommand{\Htrace}{\Lambda_{\omega_M}}
\newcommand{\Ftrace}{\Lambda_{\omega_k}}

\newcommand{\Vlap}{\Delta_{\VerT}}
\newcommand{\Hlap}{\Delta_{\HorT}}
\newcommand{\Flap}{\Delta_k}
\newcommand{\Hform}{\mu^* (F^{\nabla})}
\newcommand{\Enu}{N_{\mathfrak{su}(r)}}
\newcommand{\Cinf}{C^{\infty}}

\newcommand{\idd}{i \dbar \partial}

\newcommand{\rk}{{\rm rk}}
\newcommand{\oS}{\overline{S}}
\newcommand{\oC}{\overline{C}}
\newcommand{\oR}{\overline{R}}
\newcommand{\dneu}{\mathfrak{D}}

\newcommand{\bR}{\mathbb{R}}
\newcommand{\bT}{\mathbb{T}}
\newcommand{\grad}{\rm grad}
\newcommand{\coker}{\rm coker}

\providecommand{\brak}[1]{\left(#1\right)}
\providecommand{\sbrak}[1]{\left[#1\right]}

\renewcommand{\P}{\mathbb P}

\renewcommand{\O}{\mathcal O}

\theoremstyle{plain}
        \newtheorem{theorem}{Theorem}
        \newtheorem{proposition}[theorem]{Proposition}
        \newtheorem{lemma}[theorem]{Lemma}

\theoremstyle{definition}
        \newtheorem{definition}[theorem]{Definition}
        \newtheorem{remark}[theorem]{Remark}

\theoremstyle{plain}
        \newtheorem*{theorem*}{Theorem}
        \newtheorem*{proposition*}{Proposition}
        \newtheorem*{lemma*}{Lemma}
        \newtheorem*{corollary*}{Corollary}
        \newtheorem*{conjecture*}{Conjecture}
\theoremstyle{definition}
        \newtheorem*{definition*}{Definition}
        \newtheorem*{remark*}{Remark}
        \newtheorem*{remarks*}{Remarks}
        
\makeatletter
\def\blfootnote{\xdef\@thefnmark{}\@footnotetext}
\makeatother

\begin{document}

\title{Extremal K\"ahler metrics on projectivised vector bundles}
\author{Till {\sc Br\"onnle}}
\date{}

\maketitle

\begin{abstract}
    We prove the existence of extremal, non-csc, K\"ahler metrics on certain unstable projectivised vector bundles $\P (E) \to M$ over a cscK-manifold $M$ with discrete holomorphic automorphism group, in certain adiabatic K\"ahler classes. In particular, the vector bundles $E \to M$ under consideration are assumed to split as a direct sum of stable subbundles $E=E_1 \oplus \dots \oplus E_s$ all having different Mumford-Takemoto-slope, e.g. $\mu(E_1) > \dots > \mu(E_s)$.
\end{abstract}

{\small \tableofcontents}


\section{Introduction}\label{sec:Intro}

In this first section we shall give an overview of the problem we are considering, including an overview of related previous work, and introduce some notation.

\subsection{Previous work}\label{subsec:PrevWork}

Constant scalar curvature K\"ahler metrics (cscK in the sequel)
on projectivised vector bundles in so-called adiabatic K\"ahler classes were first constructed by Y.-J. Hong. 
In his first paper
\cite{Ho1}, Hong considered the case of a cscK base-manifold $(M,J_M,g_M, \omega_M)$
with discrete holomorphic automorphism group; and a Mumford-Takemoto-slope-stable (with respect to $[\omega_M]$) Hermitian holomorphic vector
bundle $E \rightarrow M$ endowed with a Hermitian-Einstein-connection---i.e. the Chern connection corresponding to a Hermitian-Einstein-metric---over it. 
We denote by $L^* \rightarrow \mathbb{P}(E)$ the fibrewise hyperplane
bundle $\fhb$ over $\PE$. The Hermitian-Einstein-connection
$\nabla$ on $E$ induces
a Hermitian connection $\Lconn$ on the line bundle
$L^*$; and we denote its curvature form by $\Lcurv$.
Hong then used an adiabatic
limit technique to construct a cscK-metric on $\pi : \mathbb{P}(E) \rightarrow M$ in
the K\"ahler class
\[ 
\left[ \omega_k  \right] = c_1 (\fhb)+k \pi^* [ \omega_M ], 
\]
for sufficiently large $k$. One of the crucial points of Hong's technique is, that the K\"ahler
metric
\[ 
\omega_k = \left( \frac{i}{2 \pi} \Lcurv \right) + k  \pi^*
\omega_M 
\]
gives an asymptotic approximation to a cscK-metric on $\mathbb{P}(E)$. 
It is because
of this property, that Hong can proceed by finding a formal power series
solution to the cscK-equation on $\mathbb{P}(E)$, which is $\mathcal{O}(k^{-s})$-close
(in a suitable norm)
to a genuine solution, for an integer $s>0$ arbitrarily large. Obtaining
suitable estimates for the scalar curvature map acting on K\"ahler potentials 
on $\mathbb{P}(E)$ and applying
standard elliptic-PDE-theory, Hong is able to deduce the existence of a genuine
cscK-metric on $\mathbb{P}(E)$ for $k \gg 0$ by using an implicit function theorem argument.

Hong's analysis relies essentially on the bundle $E$ being slope-stable and
therefore
also simple (i.e. it only has endomorphisms of the form $\lambda \cdot Id_E$,
with $\lambda \in \mathbb{C}^*$ and $Id_E$ the identity endomorphism).
The simplicity of the vector bundle $E$ is reflected in the linearisation
of the scalar curvature map on K\"ahler potentials
on $\mathbb{P}(E)$ having trivial co-kernel. 

In a second paper on this topic \cite{Ho2}, Hong considered the situation
of a polystable, non-simple Hermitian holomorphic vectorbundle $E=E_1 \oplus \dots
\oplus E_s$ being projectivised over a cscK base manifold $M$ with a non-trivial
Lie algebra $\mathfrak{ham} (M,J_M, \omega_M)$ of Hamiltonian Killing vector fields. The main
difference of this situation to the
above one is, that the lifting of the action of $\mathfrak{ham} (M,J_M, \omega_M)$ will
induce non-trivial Hamiltonian Killing vector fields on $\mathbb{P}(E)$. Moreover,
since $E$ is not simple anymore, the Lie-algebra $\mathfrak{g}_E$ of the
projectivisation of the automorphism group of $E$ will induce a non-trivial
action as well. Hong assumes in \cite{Ho2}, that the Futaki invariant with respect to the K\"ahler class
$[\omega_k]$ and
\[ 
\mathfrak{g}_E + (\text{the lift of}) \, \mathfrak{ham} (M,J_M, \omega_M) 
\]
on $\PE$ is zero. This assumption enables him to
solve the cscK-equation on $\mathbb{P}(E)$, without having to deal with 
any obstruction coming from a non-trivial co-kernel of its linearisation.

Another situation similar to the above ones was considered by J. Fine \cite{F}.
He treated the problem of finding a cscK-metric in adiabatic K\"ahler classes on the total space of a 
Kodaira fibration $X \rightarrow \Sigma$. Here the base is a
complex curve of high genus, and the fibres have genus at least two.
The fibres and the base admit no non-trivial holomorphic vector fields.
From this one can conclude, using the projection formula in cohomology, that the total space $X$ admits no non-trivial 
holomorphic vector fields either. Therefore, the cscK equation on $X \to \Sigma$
is solvable without any further obstructions
(the co-kernel of its linearisation consists of constant functions). The main difference in Fine's
work is, that the fibres of the Kodaira fibration have non-trivial moduli,
which leads to other difficulties in his case.
\begin{remark}The theorem, that a Hermitian holomorphic vector bundle over
a compact K\"ahler manifold admits a Hermitian-Einstein-metric 
(and thus a corresponding Hermitian-Einstein-connection) if and only if it is polystable was proven by Narasimhan-Seshadri, Donaldson and Uhlenbeck-Yau (see \cite{NS,UY,D5}). Usually, this result is referred to as the Hitchin-Kobayashi correspondence.
\end{remark}

\subsection{Introduction to the main problem}\label{subsec:MainProbl}

The situation we are considering differs from the above ones
by the fact that we will be searching for an \emph{extremal, non-cscK-metric}
on a projectivised Hermitian holomorphic vector bundle
\[
\pi : (\mathbb{P}(E),\omega_k) \to (M,\omega_M)
\]
in the K\"ahler class $\left[ \omega_k
\right]= 2 \pi c_1 (\mathcal{O}_{\mathbb{P}(E)}(1)) + k \pi^* \left[ \omega_M \right]$ for $k \gg 0$,
where $\mathcal{O}_{\mathbb{P}(E)}(1)$ is again the fibrewise hyperplane bundle
over $\mathbb{P}(E)$. The crucial difference is, that our vector bundle $E$
will be \emph{slope-unstable}. However, we will assume a certain special
structure
and look at a bundle $E$ which splits as a direct sum of slope-stable subbundles
(again, slope-stable with respect to $[\omega_M]$)
\[ E=E_1 \oplus \dots \oplus E_s,  \]
all having \emph{different slopes}. 
\begin{remark}
For convenience, we shall assume from now on that the slopes $\mu(E_i)$ satisfy
\[\mu(E_1)> \dots > \mu(E_s) .\]
\end{remark}

Since the bundles $E_i \to (M, \omega_M)$ are all stable, 
we can endow each of them with a HE-connection $\nabla_i$, i.e. the Chern-connection corresponding to a Hermitian-Einstein-metric, satisfying
\[ 
i \Lambda_{\omega_M} F^{\nabla_i} = \lambda_i \, Id_{E_i}, \qquad
\lambda_i = const. \in \mathbb{R}.  
\] 
The direct
sum of these connections will give us a (Chern) connection $\nabla=\nabla_1
\oplus \dots \oplus \nabla_s$ on $E$. As above, this 
induces a (Chern) connection $\Lconn$ on $L^* = \fhb$,
the curvature form of which we denote again by $\Lcurv$.
Similar to Hong, we will start with the K\"ahler metric
\[ 
\omega_k = i \Lcurv  + k  \pi^* \omega_M  
\]
and see that it gives us an asymptotic approximation---in a sense to be made precise later---to an
extremal, non-csc K\"ahler metric on $\mathbb{P}(E)$. Our main result is.
\begin{theorem}\label{Thm:MainResult}
Given a cscK manifold $(M,\omega_M)$ with no non-trivial
holomorphic automorphisms and a Hermitian holomorphic vector
bundle $E \rightarrow M$ splitting as a direct sum of stable subbundles
$ E=E_1 \oplus \dots \oplus E_s, $
each of them endowed with a Hermitian-Einstein-connection $\nabla_i$ and
all of them having different Mumford-Takemoto-slope; then  for $k \gg 0$ 
the projectivised vector bundle $\mathbb{P}(E) \to
 (M,\omega_M)$ has an extremal, non-csc K\"ahler-metric
in the K\"ahler class $\left[ \omega_k
\right]= 2 \pi c_1 (\mathcal{O}_{\mathbb{P}(E)}(1)) + k \pi^* \left[ \omega_M \right]$.
\end{theorem}

\paragraph{Acknowledgements.}

The work presented here forms part of the author's Ph.D.-thesis. It is a great pleasure to thank my supervisor, Simon  K. Donaldson, for the countless very useful discussions we had during the course of this work.
Also, I would like to thank Joel Fine, Dmitri Panov and Richard Thomas for useful discussions and comments, and Paul Gauduchon for his help, comments and for useful discussions.

\section{Preliminaries and background material}

We shall collect here some background material which we will need in the sequel.

\subsection{Background on extremal K\"ahler metrics}\label{subsec:EKMBackground}

The notion of an extremal K\"ahler metric on a (compact) K\"ahler manifold $(M,J,g,\omega)$ was first introduced by Calabi
in \cite{C1}. They
are defined to be the critical points of the so-called Calabi functional
\begin{equation}\label{eq:CalFunct}
C(\omega)=\int_M (Scal(\omega)- \oS)^2 \frac{\omega^n}{n!}, 
\end{equation}
in some K\"ahler-class $[\omega]$,
where $Scal(\omega)$ denotes the scalar curvature of the metric $g$ corresponding to $\omega$, and $\oS$ its average. Of course, cscK-metrics are
automatically extremal K\"ahler metrics. The converse is not always true,
the first examples of extremal, \emph{non-cscK}-metrics were constructed
by Calabi on Hirzebruch surfaces in \cite[Section~3]{C1}.
\begin{remark}
In the sequel, we will often use the K\"ahler metric $g$ on $(M,J,g,\omega)$
and its associated K\"ahler form $\omega(\cdot, \cdot)=g(J \cdot, \cdot)$ interchangeably.
\end{remark}
\begin{definition}[Extremal K\"ahler metric]\label{def:ExtrKmetr}
A K\"ahler metric $\omega \in [\omega]$ on a compact complex manifold
$(M,J)$ is called extremal (non-cscK) if it is a non-minimal critical point
of the Calabi-functional \eqref{eq:CalFunct}. 
\end{definition}
\begin{definition}[Reduced Automorphism Group]\label{def:RedAutGroup}
For a K\"ahler manifold $(M,J,g,\omega)$, we define the (identity component of the) reduced automorphism group $Aut^0_{red}(M,J)$ to be the subgroup of $Aut^0 (M,J)$, i.e. the identity component of the holomorphic automorphism group of $(M,J)$, generated by (real) holomorphic vector fields with \emph{non-trivial zero-set} on $(M,J)$. 

One can show that $Aut^0_{red}(M,J)$ is the unique linear algebraic subgroup of $Aut^0 (M,J)$ such that the quotient $Aut^0 (M,J) / Aut^0_{red}(M,J)$ is the Albanese torus of $(M,J)$.
\end{definition}
Suppose we are given a K\"ahler manifold $(M,J,g,\omega)$. We shall now choose a \emph{connected maximal compact} subgroup $G_{max}$ of the reduced automorphism group $Aut^0_{red}(M,J)$.

Then, for any $G_{max}$-invariant K\"ahler metric $\omega \in [\omega]^{G_{max}}$---where  $[\omega]^{G_{max}}$ denotes the set of $G_{max}$-invariant K\"ahler metrics (forms) in $[\omega]$---the Lie-algebra $\mathfrak{g}_{max}$ of $G_{max}$ is the space of Hamiltonian Killing vector fields 
(cf. \cite[Introduction~and~Section~1]{FM}). The key point is that the Hamiltonian Killing vector fields in $\mathfrak{g}_{max}$ remain Hamiltonian Killing vector fields as we vary $\omega$ in $[\omega]^{G_{max}}$.

\begin{definition}[Extremal vector field]\label{def:EVF}
For all $V \in \mathfrak{g}_{max}$ we define the extremal vector field \linebreak $X^{G_{max}}_{[\omega]}
\in \mathfrak{g}_{max}$, as the vector field satisfying
\begin{equation}\label{eq:EVFdef}
\mathfrak{F}(V,[\omega])=\langle X^{G_{max}}_{[\omega]}, V \rangle;
\end{equation}
where $\mathfrak{F}(V,[\omega])$ denotes the Futaki-invariant, and $\langle \cdot, \cdot \rangle$ is the Futaki-Mabuchi inner product\footnote[1]{For the definition of the Futaki-Mabuchi inner product, see \cite{FM}.} restricted to $\mathfrak{g}_{max}$. This inner product is \emph{positive definite} on $\mathfrak{g}_{max}$ (cf. \cite[Theorems~A~and~C]{FM}), which is why by duality we can define $X^{G_{max}}_{[\omega]}$ as above. The extremal vector field
$X^{G_{max}}_{[\omega]}$ depends \emph{only} 
on the K\"ahler-class and the choice of $G_{max}$,
in particular it is \emph{independent} of the choice of a $G_{max}$-invariant K\"ahler metric $\omega \in [\omega]^{G_{max}}$
(cf. \cite[Corollary~D]{FM}).
\end{definition}

\begin{remark}\label{rem:EVFtorus}
It was also shown by Futaki-Mabuchi that $X^{G_{max}}_{[\omega]}$ lies in the centre
of $\mathfrak{g}_{max}$ and generates a \emph{torus action} (cf. \cite[Theorem~F]{FM}).
\end{remark}

Calabi computed the Euler-Lagrange equation to his functional \eqref{eq:CalFunct} on a compact K\"ahler manifold $(M,J,g,\omega)$ in \cite{C1}, it is given by (again, $g$ is the metric corresponding to $\omega$)
\begin{equation}\label{eq:ELCalabieq}
\mathcal{L}_{\nabla_g Scal(g)} J=0,
\end{equation}
i.e. $\nabla_g Scal(g)$ is the real part of a holomorphic section of $T^{1,0}M$
(where $\mathcal{L}$ denotes the Lie-derivative).
Restricting to K\"ahler metrics invariant under a chosen maximal connected compact subgroup
of the reduced automorphism group, one can reduce the order of this equation as follows:
According to \cite{C1}, a K\"ahler metric $\omega \in [\omega]^{G_{max}}$ is extremal, if
\begin{equation}\label{eq:ExtrCond1}
Scal(\omega) - H (\omega)- \oS =0,  
\end{equation}
where $H (\omega)$ is a (mean-value zero) Hamiltonian for a Hamiltonian Killing vector field
(in the Lie-algebra $\mathfrak{g}_{max}$)
with respect to $\omega \in [\omega]^{G_{max}}$.

In fact, if equation \eqref{eq:ExtrCond1} is satisfied for a K\"ahler metric $\omega \in [\omega]^{G_{max}}$, then $H(\omega)$ is the (mean-value zero) Hamiltonian of the extremal vector field $X^{G_{max}}_{[\omega]}$ defined as in Definition \ref{def:EVF}. This follows from the definition of the Futaki-invariant and the Futaki-Mabuchi inner product, and the calculation
\begin{equation}\label{eq:HamFutakiCalc}
\int_M H(\omega) H_V \frac{\omega^n}{n!}
= \int_M (Scal(\omega)-\oS) H_V \frac{\omega^n}{n!}=\mathfrak{F}(V,[\omega])
 \overset{.}{=} \langle X^{G_{max}}_{[\omega]}, V \rangle,
\end{equation}
where $H_V$ denotes the (mean-value zero) 
Hamiltonian of any Hamiltonian Killing vector field $V \in \mathfrak{g}_{max}$.

\subsection{Preparatory material}\label{subsec:Prep}

Suppose we are given a rank $r:= \rk (E)$ complex holomorphic vector bundle $(E,h,\nabla)
\rightarrow (M,\omega_M)$, with Hermitian metric $h$ and Chern connection $\nabla$, 
over a (complex) $n$-dimensional K\"ahler manifold $M$. The Chern connection $\nabla$
defines a splitting
of the tangent bundle of $\mathbb{P}(E)$ in its vertical and horizontal components:
$T\mathbb{P}(E)=\VerT \oplus \HorT$, with $\VerT$ being the vertical-and
$\HorT$ being the horizontal tangent bundle.
Moreover, the Chern connection $\nabla$ induces
a Chern connection $\Lconn$ in $L^*=\fhb$; its curvature $\Lcurv$
will be an imaginary two-form. The restriction
of $i \Lcurv$ to a fibre is just the Fubini-Study metric on that fibre---induced by the Hermitian bundle metric $h$.
However, the horizontal components of $i \Lcurv$ are determined by the
curvature $F^{\nabla}$ of the connection $\nabla$ on $E$. 

We denote by $\mu^* : \mathfrak{su}(r) \rightarrow C^{\infty} 
(\mathbb{CP}^{r-1}) $ the co-moment map, which associates to every
$v \in \mathfrak{su}(r)$ its corresponding mean-value zero Hamiltonian
$\mu^* (v)$ with respect to the Fubini-Study metric. Using this co-moment map fibrewise, we get a map
$\mu^* : \Omega_M^0 (\mathfrak{su}(E)) \rightarrow C^{\infty} (
\mathbb{P}(E))$. Taking the tensor product with the pull-back map on $p$-forms $\pi^*
: \Omega_M^p \rightarrow \Omega^p_{\mathbb{P}(E)}$ extends the map $\mu^*$
to a map on $\mathfrak{su}(E)$-valued $p$-forms $\mu^* : \Omega_M^p (\mathfrak{su}(E))
\rightarrow \Omega^p_{\mathbb{P}(E)}$, and by complex linearity to $\End (E)$-valued (complex) $p$-forms. Using this notation, we get the precise
relationship between $F^{\nabla}$ and $\Lcurv$. (The following result and its proof can be found in \cite{FP}.)
\begin{proposition}[cf. Proposition 2.1 in \cite{FP}]\label{FPProp} 
With respect to the vertical-horizontal decomposition of two-forms on $\mathbb{P}(E)$:
$\Lambda^2 T \PE^* \cong \Lambda^2 \VerT^* \oplus \left(
\VerT^* \otimes \HorT^* \right) \oplus \Lambda^2 \HorT^*$, we get
\[ 
i \Lcurv = \omega_{FS} \oplus 0 \oplus \mu^* (F^{\nabla}), 
\]
where $\omega_{FS}$ restricts to the Fubini-Study metric on the fibres.
Moreover, $i \Lcurv$ is a symplectic form if and only if $\mu^* (F^{\nabla})^n$
is nowhere zero.
\end{proposition}

In the sequel, we consider the natural action of $\End (E)$ on $\PE$,
and shall now describe the associated infinitesimal action.

For any section $A$ of the vector bundle $\End (E)$ of $\mathbb{C}$-endomorphisms
of $E$, denote by $\hat A$ the vertical vector field defined as follows.
Recall that for any $x \in \PE$ with projection $\pi(x)=y$, we can identify
the vertical (real) tangent space $T_x^V \PE$ at $x$  naturally with the space \linebreak
$\Hom (x, E_y / x) \cong \Hom (x,x^{\perp})$ of $\mathbb{C}$-linear homomorphisms from the complex line $x$ to the orthogonal subspace $x^{\perp}$ to $x$ in $E_y$; where we identified $x^{\perp} \cong E_y / x$ (with the space
on the right hand side having a holomorphic structure).
\begin{remark}\label{rem:VertQuotBundle}
Since we identified $x^{\perp} \cong E_y / x$, 
we in fact defined a holomorphic structure on $\Hom (x,x^{\perp}) \ (\cong \Hom (x, E_y / x))$.
\end{remark}
\begin{definition}[Infinitesimal action induced on $\PE$ by $\End (E)$]\label{def:InfAction}
We define the vertical vector field $\hat A (x):v \mapsto Av-N_A (x) v$, for any $v$ in $x \subset
E_y$, by setting 
\begin{equation}\label{eq:HamForExpl}
N_A (x)=\frac{(Au,u)_h}{|u|_h^2};
\end{equation} 
where $u$ stands for any generator of $x$ in $E_y$ and $(\cdot, \cdot)_h$ denotes the Hermitian inner product with respect to the bundle metric $h$. 
\end{definition}
\begin{remark}
If $A$ is a constant multiple of the identity, then $\hat A$ is indeed zero,
as it should be, and $N_A$ is constant on each fibre.
\end{remark}
If $A$ is skew-hermitian, the restriction of the vertical
vector field $\hat A$ to a fibre $\mathbb{P}(E_y)$ is a Hamiltonian Killing
and real holomorphic vector field with respect to the Fubini-Study metric
on
$\mathbb{P}(E_y)$, induced by the Hermitian metric $h$ on $E$.
The fibrewise Hamiltonian of this vector field with respect to this (Fubini-Study) metric
is just $- i N_A$.
\begin{remark} Proposition \ref{FPProp}
is also true for the more general situation of the fibre 
being a general co-adjoint orbit $G/H$ (see Proposition 2.1 and Remark
2.3 in \cite{FP}).
\end{remark} 

\subsection{Future extensions}

It would be interesting to extend the results stated in Section \ref{sec:Intro} to more general
K\"ahler fibrations. Indeed, the adiabatic limit technique used in the proof of
our existence theorem is not limited to projectivised bundles, and could be applied to more general fibrations.

Suppose we are given a principal $G$-bundle $\pi: P \to (M,\omega_M)$, with
connection $\nabla$, over a cscK manifold $(M,\omega_M)$ without
holomorphic automorphisms.
We suppose the fibres of the associated bundle $X \to M$ 
to be of the form $(G/H,\omega_{G/H})$,
while the K\"ahler metric $\omega_{G/H}$ is supposed to be 
K\"ahler-Einstein.
Moreover, we assume the existence of a moment map $\mu: G/H \to \mathfrak{g}^*$,
embedding $G/H$ as an integral co-adjoint orbit. (For a detailed discussion
of the theory of (co-) adjoint orbits and existence of K\"ahler-Einstein metrics
on them, see \cite{Bes}.)

In addition to the existence of the moment map $\mu: G/H \to \mathfrak{g}^*$,
we stipulate that
the symplectic form $\omega_{G/H}$ is the curvature form of a (Chern)
connection on a Hermitian holomorphic line bundle $L \to G/H$, such that the action of
$G$ on $G/H$ lifts to a unitary action on $L$ preserving the connection.
Let $\mathcal{L}=P \times_G (G/H,L) \to X$ be the Hermitian holomorphic line bundle,
whose fibrewise restriction
is $L \to G/H$. The connection $\nabla$ enables us to combine the fibrewise
connections in $\mathcal{L}$ to give a (Chern) connection $\nabla^{\mathcal{L}}$.
Using the horizontal-vertical decomposition defined by $\nabla$, 
we obtain for the curvature of $\nabla^{\mathcal{L}}$ (cf. \cite[Remark~2.3]{FP})
\[ i F^{\nabla^{\mathcal{L}}}=\omega_{G/H} \oplus 0 \oplus \mu^* (F^{\nabla}),\]
in which $F^{\nabla}$ is the curvature form of $\nabla$, and $\mu^* (F^{\nabla})$
is defined similarly as before in Proposition~\ref{FPProp}.
Using the theory of stability and Hermitian-Einstein connections on principal
bundles of Ramanathan and Subramanian \cite{RS}, it should be possible to
formulate a criterion similar to the decomposition of the vector bundle $E
\to M$ into stable direct summands used before, for the principal fibre bundle
$P \to M$. It should
be possible to extend the main existence result for extremal metrics
on projectivised bundles, Theorem \ref{Thm:MainResult}, to this more general situation using again an
adiabatic limit technique. At the time of writing this paper the author was not able to
work everything out in detail, but these question shall be addressed in a sequel to the current paper.

\section{The formal solutions}\label{ChapApproxSol}

We are now going to construct a pointwise formal power series solution
of the extremal metric equation \eqref{eq:ExtrCond1} by adding K\"ahler potentials, found by an inductive scheme, to the metric $\omega_k$. 
Our induction scheme will be different from
the ones of Fine \cite{F} and Hong \cite{Ho1,Ho2}, since a non-trivial co-kernel
will be present in some of the linear equations we have to solve. 

However, since our induction scheme is similar in nature to the one in \cite[Section~3]{F},
we will loosely follow the structure of the exposition there.
All results obtained for the formal solutions in this section are only valid \emph{pointwise}.
Only later we will show how to establish convergence of the formal power series solutions in suitable Banach spaces.

In summary, the purpose of Section \ref{ChapApproxSol} is to produce a K\"ahler metric $\omega_{k,n}, \ n \geq 1$ with $\omega_{k,0}=\omega_k$ on $\PE \to (M,\omega_M)$, obtained by adding K\"ahler potentials  $\psi_i$ to $\omega_k$---with $\psi_i \in C^{\infty}_{\bT^s} (\PE, \bR)$, where $C^{\infty}_{\bT^s} (\PE, \bR)$ denotes the space of smooth real valued functions on $\PE$ invariant under the $\bT^s$-action induced on $\PE$ by $Id_{E_1}, \dots, Id_{E_s} \in \End (E)$ (cf. Definition \ref{def:InfAction})---such that $\omega_{k,n}$ is an \emph{approximate solution} to the extremal metric equation \eqref{eq:ExtrCond1} \emph{in the sense that} for certain constants $\oC, c_1, \dots, c_{n+1} \in \bR$,
\begin{equation}\label{eq:ExtrEqApproxDef}
Scal( \omega_{k,n} ) - Q (\omega_{k,n}) - \oC = \sum_{i=1}^{n+1} c_i k^{-i} + \mathcal{O}(k^{-n-2}),
\end{equation}
where $Q(\omega_{k,n})$ is a Hamiltonian with respect to the purely vertical part $(\omega_{k,n})_{\VerT}$ of $\omega_{k,n}$ for a (Hamiltonian Killing) vector field in the Lie algebra $\mathfrak{t}^s$ of the torus $\bT^s$ (generated by the vector fields induced by $Id_{E_1}, \dots, Id_{E_s} \in \End (E)$ on $\PE$).

In order to produce this approximate solution $\omega_{k,n}$, we have to solve three linear PDEs at each step in our induction scheme. As explained in Subsection \ref{subsubsec:Splitting} below, the errors we have to correct in order to successively adjust a given approximate solution to a higher order approximate solution live in three function spaces $\Enu, \ R, \ C^{\infty}(M)$. The $\Enu$-parts of the errors are corrected by perturbing the hermitian bundle metric $h$ on $E \to M$ \emph{and} the (Hamiltonian Killing) vector field which corresponds, with respect to the purely vertical part $(\omega_{k,n})_{\VerT}$ of $\omega_{k,n}$, to the Hamiltonian $Q(\omega_{k,n})$ at each step (for the details, see Subsection \ref{subsec:Epart}).
The $R$-parts of the errors are corrected by adjusting $\omega_{k,n}$ at a certain step in the induction scheme by a $\bT^s$-invariant K\"ahler potential which is $L^2$-orthogonal to the function space $\Enu$ (for the details, see Subsection \ref{subsec:Rpart}). Finally, the $C^{\infty}(M)$-parts of the errors are corrected by adjusting the K\"ahler form $\omega_M$ on the base manifold $M$ by suitable K\"ahler potentials (for the details, see Subsection \ref{subsec:Cpart}).

\subsection{The first order approximate solution}\label{sec:1stOrdApprox}

We shall now compute the scalar curvature of $\omega_{k,0}=\omega_k = 
i \Lcurv + k \cdot \pi^* \omega_M$. 
But first, we will need some more terminology.

\

Splitting the trace $\Ftrace$ with respect to $\omega_k$ up into vertical
and horizontal parts motivates the following definitions.

\begin{definition}\label{def:traces}
The \emph{vertical trace} is defined by
\[ \Vtrace \alpha=\frac{(r-1) \alpha \wedge \omega^{r-2}_{FS}}{\omega^{r-1}_{FS}},\]
for $\alpha \in \Lambda^2 \VerT^*$, where the quotient is taken in the line $\det \mathcal{V}^*$ (as $\omega_{FS} \in \Lambda^2 \VerT^*$ and $\rk (\mathcal{V})=r-1$, $r=\rk (E)$,
this is well-defined). The \emph{horizontal trace} is defined by
\[ \Htrace \alpha= \frac{n \alpha \wedge \omega_M^{n-1}}{\omega_M^n},\]
for $\alpha \in \Lambda^2 \HorT^*$, where the quotient is taken in the line $\det \mathcal{H}^*$  (as $\omega_M \in \Lambda^2 \HorT^*$ and $\rk (\mathcal{H})=\dim (M)=n$,
this is also well-defined).
\end{definition}
\begin{lemma}\label{lem:VHtrace}
Let $\alpha \in \Lambda^2 T \PE^*$, then
\[\Ftrace \alpha= \Vtrace \brak{\alpha}_{\VerT} + k^{-1}
\Htrace \brak{\alpha}_{\HorT} + \mathcal{O}(k^{-2}),\]
where $(\alpha)_{\HorT}$ and $(\alpha)_{\VerT}$ denote the purely
horizontal and purely vertical components of the form~$\alpha$.
\end{lemma}
\begin{proof}
The result is obtained by computing:
\begin{align*}
\Ftrace \alpha = & \frac{(n+r-1) \alpha \wedge \omega_k^{n+r-2}}{\omega_k^{n+r-1}}\\
 = & \frac{(r-1) \brak{\alpha}_{\VerT} \wedge \omega_{FS}^{r-2}\wedge \brak{\Hform
+k \omega_M}^n }{\omega_{FS}^{r-1}\wedge \brak{\Hform +k \omega_M}^n}\\
& + \frac{n \brak{\alpha}_{\HorT} \wedge \omega_{FS}^{r-1}\wedge \brak{\Hform
+k \omega_M}^{n-1} }{\omega_{FS}^{r-1}\wedge \brak{\Hform +k \omega_M}^n}\\
 = & \Vtrace \brak{\alpha}_{\VerT} + k^{-1} \Htrace \brak{\alpha}_{\HorT}
+ \mathcal{O}(k^{-2});
\end{align*}
where in the last equality we expanded the second fraction in a power series in terms of $k^{-1}$,
and absorbed the terms containing $\Hform$ into the $\mathcal{O}(k^{-2})$-terms.
\end{proof}
\begin{definition}\label{def:vhLapl2D}
The vertical and horizontal Laplacians (on functions) are defined by
\[ \Vlap f = \Lambda_{\omega_{FS}} \brak{\idd f}_{\VerT}, \]
and
\[ \Hlap f = \Lambda_{\omega_M} \brak{\idd f}_{\HorT}. \]
The fibrewise restriction of $\Vlap$ is the Laplacian on a fibre
determined by $\omega_{FS}$. Whereas on functions pulled back from the base,
$\Hlap$ is the Laplacian defined by $\omega_M$.
\end{definition} 
\begin{lemma}\label{lem:LaplSplit}
The $\omega_k$-Laplacian on functions, denoted by $\Flap$, satisfies
\[\Flap f= \Vlap f + k^{-1} \Hlap f + \mathcal{O}(k^{-2}).\]
\end{lemma}
\begin{proof}
This follows immediately from the decomposition of $\Ftrace$ obtained
in Lemma \ref{lem:VHtrace}.
\end{proof}
\begin{lemma}\label{1storder}
For the first order approximate 
solution $\omega_k$ we get
\begin{equation}\label{1stLin} 
Scal(\omega_k)=\oC+k^{-1} \left(
Scal(\omega_M) + b \mu^* (\Lambda_{\omega_M} F^{\nabla})  \right) + \mathcal{O}(k^{-2}), 
\end{equation}
for some constants $\oC,b$ depending only on $r$; and $\mu^*$ is again the map defined at the beginning of Section \ref{subsec:Prep}.
\end{lemma}
\begin{proof} We have the short exact sequence of vector bundles on $\PE$
\[0 \to \VerT \to T\mathbb{P}(E) \to \HorT \to 0. \]
Therefore, we have the $C^{\infty}$-splitting $T \PE = \VerT \oplus \HorT$ (as already mentioned
in Section \ref{subsec:Prep} above). This is not a \emph{holomorphic} splitting and in general $\HorT$ defined via this splitting won't be a holomorphic subbundle of $T \PE$. However, as the vertical tangent bundle $\VerT$ is a holomorphic subbundle
of $T \PE$, the quotient bundle $T \PE / \VerT$ is also a holomorphic  vector bundle. Moreover,
we have the $C^{\infty}$-isomorphism $\HorT \cong T \PE / \VerT$, and for the calculation below we shall use this identification and consider $\HorT$ as a holomorphic vector bundle.

Thus we have the isomorphism $K_{\mathbb{P}(E)} \cong \Lambda^{r-1} \VerT^* \otimes
\Lambda^n \HorT^*$ of holomorphic line-bundles. Hence the Ricci form
\[\rho_k = i F^{\Lambda^{r-1} \VerT^*} + i F^{\Lambda ^n \HorT^*}, \]
where $F^{\Lambda^{r-1} \VerT^*},F^{\Lambda ^n \HorT^*}$ are the curvature
forms of $\Lambda^{r-1} \VerT^*, \Lambda ^n \HorT^*$.

With $\Lambda^{r-1} (\VerT^*) \cong \mathcal{O}_{\PE}(-r) \otimes (\det E)^{-1}$,
we see that $\omega_k= \omega_{FS} \oplus \Hform + k \omega_M$ induces a metric $h_{\VerT}$ on $\Lambda^{r-1} \VerT^*$
which is determined by the fibrewise Fubini-Study metrics. So, $h_{\VerT}$
is the $r$-th power of the metric on $\fhb$ (which is induced by the metric $h$ on $E$),
hence its curvature is just $r \Lcurv$.

The curvature $F^{\Lambda ^n \HorT^*}$ of $\Lambda^n \HorT^*$ depends on $k$,
as the metric on $\HorT$ corresponds to the K\"ahler-form $\Hform +k \omega_M$. 
Denote by
$\rho_M$ the Ricci form (pulled back\footnote[1]{We won't denote the pullback of functions, forms, etc. explicitly.}
to $\PE$), i.e. the curvature form of the Chern connection
on $K^*_M$, the anti-canonical line bundle of $M$,
determined by $\omega_M$. Since the horizontal tangent bundle $\HorT$ projects to the tangent bundle $TM$ of the base manifold $M$, we will identify $\Lambda^n \HorT^* \cong \pi^* K^*_M$
as holomorphic line-bundles.

The ratio of the top exterior powers of the two K\"ahler forms $\Hform +k \omega_M$ and $\omega_M$
gives us the ratio of the corresponding metrics on the (holomorphic) line bundle $\Lambda^n \HorT^*$.
By general theory, we then know that $iF^{\Lambda ^n \HorT^*}$ and $\rho_M$ are related by
\begin{align*}
iF^{\Lambda ^n \HorT^*}-\rho_M & =\idd \log \left( \frac{\brak{\Hform +k \omega_M}^n}{\omega_M^n} \right)\\
& = \idd \log \brak{k^n + \mu^* (\Lambda_{\omega_M} F^{\nabla}) k^{n-1}+ \mathcal{O}(k^{n-2})}.
\end{align*}
Thus, the Ricci form of $\omega_k$ is given by 
\begin{align*}
\rho_k & = i F^{\Lambda^{r-1} \VerT^*} + i F^{\Lambda ^n \HorT^*}\\ 
& = r i \Lcurv + i F^{\Lambda ^n \HorT^*}\\
& = r i \Lcurv + \rho_M + i \dbar \partial \log \brak{k^n + \mu^* (\Lambda_{\omega_M} F^{\nabla}) k^{n-1}+ \mathcal{O}(k^{n-2})}\\
& = r i \Lcurv + \rho_M + \underbrace{i \dbar \partial \log k^n}_{=0} +
i \dbar \partial \log \brak{1 + \mu^* (\Lambda_{\omega_M} F^{\nabla}) k^{-1}+ \mathcal{O}(k^{-2})}.
\end{align*}
Using the power series expansion $ \log (1+x) = \sum_{i=1}^{\infty} (-1)^{i+1}
\frac{x^i}{i}, \ |x|<1$ (which is possible since $k \gg 0$), we obtain
\begin{equation}\label{eq:Riccik}
\rho_k= r i \Lcurv + \rho_M + k^{-1} 
i \dbar \partial (\mu^* (\Lambda_{\omega_M} F^{\nabla}))+ \mathcal{O}(k^{-2}) 
\end{equation}
\[
= r \omega_{FS} + r \Hform+ \rho_M + k^{-1} 
i \dbar \partial (\mu^* (\Lambda_{\omega_M} F^{\nabla}))+ \mathcal{O}(k^{-2}) 
\]
Using Lemmas \ref{lem:VHtrace}, \ref{lem:LaplSplit},
and the fact that the Ricci-form of the Fubini-Study metric induced on the fibres by $\fhb$ is $\rho_{FS}=r \omega_{FS}$,
we get by taking the trace of $\rho_k$ with $\omega_k$
\[Scal(\omega_k)= Scal(\omega_{FS}) + k^{-1} \brak{r \mu^* (\Lambda_{\omega_M} F^{\nabla}) +Scal(\omega_M)+ \Vlap
(\mu^* (\Lambda_{\omega_M} F^{\nabla}))}+ \mathcal{O}(k^{-2}).\]
Moreover, using that $\mu^* (\Lambda_{\omega_M} F^{\nabla})$ is in the first eigenspace of $\Vlap$---with first eigenvalue $\nu_1=2 r$---we get
\[Scal(\omega_k)= Scal(\omega_{FS}) + k^{-1} \brak{Scal(\omega_M)+b \mu^* (\Lambda_{\omega_M} F^{\nabla})}+ \mathcal{O}(k^{-2}),\]
with some constant $b$ depending only on $r$.
Setting $\oC:=Scal(\omega_{FS})=2 r (r-1)$ gives us equation~\eqref{1stLin}.
\end{proof}

\subsubsection{Splitting of function spaces on $\PE$}\label{subsubsec:Splitting}

The space of smooth functions $C^{\infty}(\mathbb{P}(E))$
on $\mathbb{P}(E) \to M$ splits as follows
\[C^{\infty}(\mathbb{P}(E))= C^{\infty}_0 (\mathbb{P}(E))\oplus 
C^{\infty}(M),\] 
where $C^{\infty}(M)$ are the smooth functions pulled back from the base;
and the space $C^{\infty}_0(\mathbb{P}(E))$ of smooth functions of fibrewise mean-value zero
splits further into 
\[C^{\infty}_0(\mathbb{P}(E))=\Enu \oplus R,\]
where functions in $\Enu$ restrict to mean-value zero Hamiltonians for an isometry of a fibre with respect to the Fubini-Study metric,
while the functions in $R$ are $L^2$-orthogonal to $\Enu$ and the
constant functions. 
In total we get a splitting  into three function spaces
\begin{equation}\label{eq:FuncSpace Split}
C^{\infty}(\mathbb{P}(E))=\Enu \oplus R \oplus C^{\infty}(M),
\end{equation}
which depends on the Fubini-Study metric induced on the fibres of $\PE \to M$, and thus on the 
Hermitian bundle metric $h$ and the corresponding Chern connection $\nabla_h$ on $E \to M$. 

In order to perturb $\omega_k$ to a higher order approximation of an
extremal K\"ahler metric, we will have to deal with errors living in these
three function spaces. As already mentioned above, these errors will be corrected by solving linear PDEs.

\subsection{The second order approximate solution}\label{sec:2ndApprox}

\subsubsection{Linearisation formulas}\label{subsubsec:LinFormulas}

The next lemma is the same as \cite[Lemma~2.1]{F}, about the linearisation of the scalar
curvature map on K\"ahler potentials on a K\"ahler manifold $(M,J,g,\omega)$; similar formulas can also be found in \cite[Section~2]{LS}.
We are considering the map $Scal: \phi \mapsto Scal(\omega_{\phi})$, with
$\omega_{\phi}:=\omega+i \dbar \partial \phi$; which is defined on some
open set $U \subset C^{\infty} (M)$.
\begin{lemma}[cf. Lemma 2.1 in \cite{F}]\label{lem:LinScalMap}
On a K\"ahler manifold $(M,J,g,\omega)$,
let $V$ denote the $L^p_{m+4}$-Sobolev completion of $U \subset C^{\infty} (M)$. The scalar curvature map on K\"ahler potentials, $Scal$, extends to a smooth
map $Scal: V \to L^p_m$ whenever $(m+2)p-2n>0$, where $n=\dim_{\mathbb{C}}
M$ is the dimension of the underlying manifold $M$. Its linearisation at $0 \in V$
is given by
\begin{equation}\label{eq:LinScalMap}
L_{Scal, \omega}(\phi)= \left( \Delta^2- Scal(\omega_0) \Delta \right) \phi + n (n-1)
\frac{i \dbar \partial \phi \wedge \rho \wedge \omega^{n-2}}{\omega^n},
\end{equation}
where $\rho$ denotes the Ricci-form of $\omega$.
\end{lemma}

Frequently, we will have to use another form of the linearisation of the scalar curvature map on K\"ahler potentials.
Using a Weitzenb\"ock-type formula for the Lichnerowicz-operator $\LichOp$,
equation \eqref{eq:LinScalMap} can also be (re-)written as in the following Lemma.
(Rigorous proofs of the two lemmas stated below, with a slightly different convention for the Laplacian and scalar curvature, can be found in \cite[Section~2]{LS}.)
\begin{lemma}\label{lem:LinScalNeu}
On a K\"ahler manifold $(M,J,g,\omega)$,
the linearisation $L_{Scal, \omega}$ of the scalar curvature map on K\"ahler potentials is given by
\begin{equation}\label{eq:LinScalNeu}
L_{Scal, \omega} ( \phi ) = \LichOp \phi +  \frac{1}{2} \nabla Scal \cdot \nabla \phi,
\end{equation}
where the gradient and inner product in the last summand are taken with respect to the metric $g$ corresponding to $\omega$.
\end{lemma}

In the same vein, we obtain the analogous result for the linearisation
of the extremal metric operator $Scal(\omega)-H(\omega)-\oS$.
\begin{lemma}\label{lem:LinExNeu}
On a K\"ahler manifold $(M,J,g,\omega)$,
the linearisation $L_{Extr, \omega}$ of the extremal metric operator $Scal(\omega)-H(\omega)-\oS$ on K\"ahler potentials \emph{invariant} under the chosen maximal connected compact subgroup $G_{max}$ of the reduced automorphism group $Aut^0_{red}(M,J)$, is given by
\begin{equation}\label{eq:LinExNeu}
L_{Extr, \omega}(\phi)=\LichOp (\phi)+ \frac{1}{2} \nabla Scal(\omega) \cdot
\nabla \phi- \frac{1}{2} \nabla H(\omega) \cdot \nabla \phi,
\end{equation}
where the gradients and inner products are taken with respect 
to the metric $g$ corresponding to $\omega$.
Here, $H(\omega)$ is the Hamiltonian with respect to $\omega$ of the extremal vector field determined by $G_{max}$ and $[\omega]$ (cf.  Definition \ref{def:EVF}).
\end{lemma}
Hence if we linearise the extremal metric operator $Scal(\omega)-H(\omega)-\oS$, \emph{at an extremal metric},
the last two summands in equation \eqref{eq:LinExNeu} drop out as the metric
already satisfies equation~\eqref{eq:ExtrCond1}, and
we get the Lichnerowicz-operator $\LichOp (\phi)$.

\subsubsection{Correcting the $\Enu$-part}\label{subsec:Epart}

The $\mathcal{O}(k^{-2})$-error in equation \eqref{1stLin}, which we
will denote by $\eta_{\mathcal{O}(k^{-2})}$, splits according to the splitting \eqref{eq:FuncSpace Split} of the
function space $C^{\infty}_0(\mathbb{P}(E))$,
\[ 
\eta_{\O (k^{-2})} = \eta_{\O (k^{-2}), \Enu} + \eta_{\O (k^{-2}), R} + \eta_{\O (k^{-2}), C^{\infty}(M)}.
\]
In order to get rid of the $\eta_{\O (k^{-2}), \Enu}$-part of the $\mathcal{O}(k^{-2})$-error $\eta_{\O (k^{-2})}$,
we will employ a technique which
involves perturbing the Hermitian metric $h$ on $E \to M$ by a suitable
Hermitian bundle endomorphism. In the current section, it becomes important that $\mu^* (\Lambda_{\omega_M} F^{\nabla})$ depends on the (Hermitian) bundle metric $h$. For this reason, we shall write $\mu^* (\Lambda_{\omega_M} F^{\nabla})=\mu^* (h,\Lambda_{\omega_M} F^{\nabla_h})$---emphasising on the $h$-dependence of the map $\mu^*$ and the Chern connection $\nabla=\nabla_h$ on $E \to M$---from now on.

\

Remember equation \eqref{1stLin} which says that the scalar curvature of $\omega_k$ is given by
\begin{align}
Scal(\omega_k)= & \oC+k^{-1} \left(
Scal(\omega_M) + b \mu^* (h,\Lambda_{\omega_M} F^{\nabla_h})  \right) \label{eq:EnuScalwithhDep} \\
& + k^{-2} \left( \eta_{\O (k^{-2}), \Enu} + \eta_{\O (k^{-2}), R} + \eta_{\O (k^{-2}), C^{\infty}(M)} \right) + \mathcal{O}(k^{-3}) \nonumber,
\end{align}
where we explicitly wrote out the $\mathcal{O}(k^{-2})$-error.

{\bf Step 1.} We are going to change $h$ to a new bundle metric $h':=h \brak{1+k^{-1} V}$, where $V$
is a Hermitian bundle endomorphism, i.e. the two metrics $h,h'$ are related via
\[ \left((1+k^{-1}V) (\cdot), \cdot\right)_h=(\cdot, \cdot)_{h'}.\]
This change of the metric $h$ will cause two types of changes in $\mu^* (h,\Lambda_{\omega_M} F^{\nabla_h})$. Namely, the one caused by the $h$-dependence of $\mu^*$ itself---indicated by the first argument of $\mu^* (\cdot, \cdot)$; and the other comes from varying $\Lambda_{\omega_M} F^{\nabla_h}$---the second argument of $\mu^* (\cdot, \cdot)$ in which it is actually linear. We write the total variation $\delta \mu^* (h,\Lambda_{\omega_M} F^{\nabla_h})$ as the sum of these two variations
\[
\delta \mu^* (h,\Lambda_{\omega_M} F^{\nabla_h})= \delta_h \mu^* (h,\Lambda_{\omega_M} F^{\nabla_h})
+ \delta_{\Lambda_{\omega_M} F^{\nabla_h}} \mu^* (h,\Lambda_{\omega_M} F^{\nabla_h}).
\]
In order to correct the $\eta_{\O (k^{-2}), \Enu}$-part of the $\mathcal{O} (k^{-2})$-error, we set
\begin{equation}\label{eq:EnuVCorr}
\delta_{\Lambda_{\omega_M} F^{\nabla_h}} \mu^* (h,\Lambda_{\omega_M} F^{\nabla_h}) =- \eta_{\O (k^{-2}), \Enu},
\end{equation}
which will give us an equation for $V$.

For the Hamiltonian of the (real holomorphic) Hamiltonian Killing vector field $\widehat{\Lambda_{\omega_M} F^{\nabla_h}}= \widehat{-i \sum_{p=1}^s \lambda_p Id_{E_p}}$ (defined as in Definition \ref{def:InfAction}) with respect to the metric $\omega_{FS}(h')$---which is the purely vertical part of $\omega_0 (h')$ with respect to the perturbed bundle metric $h'$---we will use the abbreviation $\mu^* (h',\Lambda_{\omega_M} F^{\nabla_h})=\mu^* (h,\Lambda_{\omega_M} F^{\nabla_h})+\delta_h \mu^* (h,\Lambda_{\omega_M} F^{\nabla_h})$.

{\bf Step 2.} Using the formula
\[h'^{-1}=\brak{1- k^{-1} V} h^{-1}+ \mathcal{O}(k^{-2}),\]
where $h^{-1}, h'^{-1}$ denote the (local) inverses of the metrics $h,h'$,
we are ready to compute the change $\delta_{\Lambda_{\omega_M} F^{\nabla_h}} \mu^* (h,\Lambda_{\omega_M} F^{\nabla_h})$ of $\mu^* (h,\Lambda_{\omega_M} F^{\nabla_h})$.
Since locally the curvature of the Chern connection $\nabla_h$ is given by $F^{\nabla_{h'}}= \dbar (h'^{-1} \partial h')$,
\begin{align*}
h'^{-1} \partial h' & = h^{-1} \partial h + k^{-1}
\brak{\partial V +\sbrak{h^{-1} \partial h,V}}+ \mathcal{O}(k^{-2})\\
& = h^{-1} \partial h + k^{-1} \partial_h V + \mathcal{O}(k^{-2}),\\
F^{\nabla_{h'}} & = F^{\nabla_{h}} + k^{-1} \dbar \partial_h V + \mathcal{O}(k^{-2}),
\end{align*}
where $\partial_h$ is the $(1,0)$-part of the Chern connection of the bundle metric $h$
(for the $(0,1)$-part we have $\dbar_h=\dbar$, thus we dropped the index).
Contracting, using the K\"ahler identity $\partial_h^* = i [ \Lambda, \dbar]$,
gives
\begin{equation}\label{eq:CurvChange}
\Htrace F^{\nabla_{h'}}= \Htrace F^{\nabla_{h}} - k^{-1} i \Delta_{\partial_h} V +\mathcal{O}(k^{-2}),
\end{equation}
where $\Delta_{\partial_h}$ denotes the $\partial_h^* \partial_h$-Laplacian
acting on endomorphisms
(determined by $h$).

Hence for $\mu^* (h',\Lambda_{\omega_M} F^{\nabla_{h'}})$ we get
\begin{align*}
\mu^* (h',\Lambda_{\omega_M} F^{\nabla_{h'}})= & \mu^* (h,\Lambda_{\omega_M} F^{\nabla_{h}})+\delta_h \mu^* (h,\Lambda_{\omega_M} F^{\nabla_h})- k^{-1} \mu^* (h,i \Delta_{\partial_h} V)+ \mathcal{O}(k^{-2}) \\
= & \mu^* (h',\Lambda_{\omega_M} F^{\nabla_h}) - k^{-1} \mu^* (h,i \Delta_{\partial_h} V)+ \mathcal{O}(k^{-2}).
\end{align*}
Therefore, after changing $h$ to $h'=h \brak{1+k^{-1} V}$, the scalar curvature of $\omega_k (h')$ is
\begin{align*}
Scal(\omega_k (h'))= & \oC+k^{-1} \left(
Scal(\omega_M) + b \mu^* (h',\Lambda_{\omega_M} F^{\nabla_h})  \right)\\
& + k^{-2} \left( -b \mu^* (h,i \Delta_{\partial_h} V) +\eta_{\O (k^{-2}), \Enu} + \eta_{\O (k^{-2}), R} + \eta_{\O (k^{-2}), C^{\infty}(M)} \right)
+ \mathcal{O}(k^{-3}). 
\end{align*}
Hence equation \eqref{eq:EnuVCorr} becomes
\[
b \mu^* (h,i \Delta_{\partial_h} V) = \eta_{\O (k^{-2}), \Enu}.
\]
Writing $\mu^* (h,U):= \eta_{\O (k^{-2}), \Enu}$ for some skew-hermitian endomorphism $U$, which is possible since $\eta_{\O (k^{-2}), \Enu} \in \Enu$---the space of mean-value zero Hamiltonians for isometries on the fibres of $\PE \to M$, gives
\begin{equation}\label{eq:EnuError}
bi \Delta_{\partial_h} V=U.
\end{equation}

{\bf Step 3.} In this last step, we solve equation \eqref{eq:EnuError}.

The Laplacian $\Delta_{\partial_h}$ has a non-trivial (co-)kernel in $\End (E)$.
Since the vector bundle we consider splits as a direct sum of stable subbundles
of different slopes $E=E_1 \oplus \dots \oplus E_s$, this (co-)kernel is generated 
by the identity endomorphisms $Id_{E_1}, \dots, Id_{E_s}$. 
Therefore, the projection of $U$ to $\coker_{\End (E)} \Delta_{\partial_h}$ can be written as
\[
\proj_{\coker_{\End (E)} \Delta_{\partial_h}} (U)=i (\gamma_1 Id_{E_1}+ \dots + \gamma_s Id_{E_s}),
\]
for suitably chosen $\gamma_1, \dots, \gamma_s \in \bR$. Subtracting $\proj_{\coker_{\End (E)} \Delta_{\partial_h}} (U)$ from the right hand side of equation \eqref{eq:EnuError}, we can now solve (using standard elliptic PDE-theory)
\begin{equation}\label{eq:EneuVProj}
bi \Delta_{\partial_h} V =U- \proj_{\coker_{\End (E)} \Delta_{\partial_h}} (U)=U- i (\gamma_1 Id_{E_1}+ \dots + \gamma_s Id_{E_s})
\end{equation}
for $V$. Thus, we have found the desired bundle endomorphism $V$ and can therefore correct the $\eta_{\O (k^{-2}), \Enu}$-error by setting $h'=h(1+k^{-1} V)$. 

However, subtracting $\proj_{\coker_{\End (E)} \Delta_{\partial_h}} (U)$ from the right hand side of equation \eqref{eq:EnuError}, we have to add it back on to the right hand side of equation \eqref{eq:EnuScalwithhDep}. In fact, with $U$ given by $\mu^* (h,U):= \eta_{\O (k^{-2}), \Enu}$, re-writing equation \eqref{eq:EnuScalwithhDep} as
\begin{align*}
Scal(\omega_k)= & \oC+k^{-1} \left(
Scal(\omega_M) + b \mu^* (h,\Lambda_{\omega_M} F^{\nabla_h})  \right) \\
& + k^{-2} \mu^* \left( h, \left( \proj_{\coker_{\End (E)} \Delta_{\partial_h}} (U) - \proj_{\coker_{\End (E)} \Delta_{\partial_h}} (U) \right) \right) \\
& + k^{-2} \left( \eta_{\O (k^{-2}), \Enu} + \eta_{\O (k^{-2}), R} + \eta_{\O (k^{-2}), C^{\infty}(M)} \right) + \mathcal{O}(k^{-3})
\end{align*}
leaves it unchanged (because the terms in the second line add to zero). Since $\proj_{\coker_{\End (E)} \Delta_{\partial_h}} (U)=i (\gamma_1 Id_{E_1}+ \dots + \gamma_s Id_{E_s})$, using that $\mu^*(\cdot, \cdot)$ is linear in its second argument, one can further re-write this as
\begin{align}
Scal(\omega_k)= & \oC+k^{-1} \left(
Scal(\omega_M) + \mu^* \left( h,b \Lambda_{\omega_M} F^{\nabla_h}+ k^{-1} i (\gamma_1 Id_{E_1}+ \dots + \gamma_s Id_{E_s}) \right)  \right) \label{eq:EnuScalReWrite} \\
& - k^{-2} \mu^* \left( h, \left( \proj_{\coker_{\End (E)} \Delta_{\partial_h}} (U) \right) \right) \nonumber \\
& + k^{-2} \left( \eta_{\O (k^{-2}), \Enu} + \eta_{\O (k^{-2}), R} + \eta_{\O (k^{-2}), C^{\infty}(M)} \right) + \mathcal{O}(k^{-3}). \nonumber
\end{align}
Therefore, using $b \Lambda_{\omega_M} F^{\nabla_h}= -bi \sum_{p=1}^s \lambda_p Id_{E_p}$, the ``trick'' we used to solve equation \eqref{eq:EnuError}---i.e. adding and subtracting $\proj_{\coker_{\End (E)} \Delta_{\partial_h}} (U)$---can be interpreted as changing the weights of the Hamiltonian $\bT^s$-action,
induced by $Id_{E_1}, \dots, Id_{E_s} \in \End (E)$ on $\PE$ (as in Definition \ref{def:InfAction}), since $ b \mu^* (h,\Lambda_{\omega_M} F^{\nabla_h})$ becomes
\[
\mu^* \left( h, i \sum_{p=1}^s (-b \lambda_p + k^{-1}  \gamma_p) Id_{E_p} \right).
\]
Using the re-written version of equation \eqref{eq:EnuScalwithhDep}, equation \eqref{eq:EnuScalReWrite}, we can go through the steps 1--3 explained above again via setting $h'=h(1+ k^{-1}V)$ and solving equation \eqref{eq:EneuVProj} for $V$, which gives us
\begin{align}
Scal(\omega_k (h'))= & \oC+k^{-1} \left(
Scal(\omega_M) + \mu^* \left( h', i \sum_{p=1}^s (-b \lambda_p + k^{-1}  \gamma_p) Id_{E_p} \right)  \right) \label{eq:EnuCorrScalReWrite} \\
& + k^{-2} \left( \eta_{\O (k^{-2}), R} + \eta_{\O (k^{-2}), C^{\infty}(M)} \right) + \mathcal{O}(k^{-3}). \nonumber
\end{align}

By Proposition \ref{FPProp}
and our definition of $\omega_k$,
\[\omega_k (h)= \omega_{FS}(h) + \mu^* (h, F^{\nabla_h}) + k \omega_M,\]
where we emphasised on the $h$-dependence of the first two summands.
These first two summands are representatives of the class $c_1(\fhb)$,
and therefore for any two metrics $h,h'=h \brak{1+k^{-1} V}$ 
on $E$ they are cohomologous. By general theory, the two metrics $\omega_k (h), \omega_k (h')$ are related by
\[
\omega_k (h') - \omega_k (h) = k^{-1} \idd \left( \sum_{d=0}^{\infty} k^{-d} \zeta_{d, \mathcal{O}(k^{-2}),\Enu} \right)=: k^{-1} \idd \phi_{\mathcal{O}(k^{-2}),\Enu},
\]
where it is crucial (in particular for the analysis done later in Section \ref{Chap6}) to observe that
\[
\omega_k (h') - \omega_k (h) = \O (k^{-1}).
\]
Therefore, the same effect as varying the metric $h$ on the bundle $E$
can also be achieved by adding 
$k^{-1} \idd \phi_{\mathcal{O}(k^{-2}),\Enu}$---where the K\"ahler potential $\phi_{\mathcal{O}(k^{-2}),\Enu}$ depends on powers of $k^{-1}$---to $\omega_k \ (=\omega_k (h))$. Clearly the K\"ahler potential $\phi_{\mathcal{O}(k^{-2}),\Enu} \in C^{\infty} (\PE, \bR)$ is $\bT^s$-invariant, i.e. $\phi_{\mathcal{O}(k^{-2}),\Enu} \in C^{\infty}_{\bT^s} (\PE, \bR)$, which follows directly from the fact that the metrics $\omega_k (h),\omega_k (h')$ and also their difference are $\bT^s$-invariant. 
Using $\omega_k (h') = \omega_k (h) + k^{-1} \idd \phi_{\mathcal{O}(k^{-2}),\Enu}$, we write as the conclusion of this section
\begin{align}
Scal(\omega_k +k^{-1} \idd \phi_{\mathcal{O}(k^{-2}),\Enu}) = & \oC+k^{-1} \left(
Scal(\omega_M) + \mu^* \left( h', i \sum_{p=1}^s (-b \lambda_p + k^{-1}  \gamma_p) Id_{E_p} \right)  \right) \nonumber \\
& + k^{-2} \left( \eta_{\O (k^{-2}), R} + \eta_{\O (k^{-2}), C^{\infty}(M)} \right) + \mathcal{O}(k^{-3}). \label{eq:EnuCorrScalConclus}
\end{align}

\subsubsection{Correcting the $R$-part}\label{subsec:Rpart}

Using the results in Section \ref{subsubsec:LinFormulas}, we get.
\begin{lemma}\label{1storderLin}
Denote again by $L_{Scal,\omega_k}$ the \emph{formal linearisation} of the scalar curvature map
on K\"ahler potentials defined by $\omega_k$. Then
\[ L_{Scal, \omega_k}= L_{Scal, F} + \mathcal{O}(k^{-1}), \]
where $L_{Scal, F}$ is the fibrewise linearisation of the scalar curvature map
(on K\"ahler potentials), i.e. $L_{Scal, F}(\phi)$ is defined as the change in scalar curvature determined by adding $\idd (\phi |_{fibre})$ on the Fubini-Study metrics induced on the fibres of $\PE \to M$.
\end{lemma}
\begin{proof} For the linearisation $L_{Scal, \omega}$ of the scalar curvature map on K\"ahler
potentials on a K\"ahler manifold $(M,J,g,\omega)$, $Scal:\phi \in C^{\infty} \mapsto
Scal(\omega_{\phi})$, $\omega_{\phi}=\omega+ \idd \phi$, we 
have by equation \eqref{eq:LinScalMap}
\begin{equation}\label{linscalmap} 
L_{Scal, \omega}(\phi)= \left( \Delta^2 - Scal(\omega_0) \Delta
\right) \phi + n(n-1) \frac{i \dbar \partial \phi \wedge \rho \wedge \omega^{n-2}}{\omega^n},
\end{equation}
where $\rho$ is the Ricci-from of the K\"ahler metric induced by $\omega$.
Applying this to the scalar curvature map on K\"ahler potentials on $(\mathbb{P}(E),\omega_k)$
gives us
\[L_{Scal, \omega_k} (\phi) = \left( \Delta_k^2 - Scal(\omega_k) \Delta_k \right) \phi +
(n+r-1)(n+r-2) \frac{i \dbar \partial \phi \wedge \rho_k \wedge \omega_k^{(n+r
-3)}}{\omega_k^{(n+r-1)}},
\]
where as above, $\Delta_k$ is the Laplacian defined by $\omega_k$.
Using equation \eqref{1stLin}, and equation \eqref{eq:Riccik} for $\rho_k$
together with Lemma \ref{lem:LaplSplit}, gives
\begin{align*}
L_{Scal, \omega_k} (\phi)  = & \brak{\Vlap^2-Scal(\omega_{FS}) \Vlap}\phi \\
&+ (r-1)(r-2)\frac{\brak{i \dbar \partial \phi}_{\VerT} \wedge \brak{\rho_k}_{\VerT}
\wedge \omega_{FS}^{r-3} \wedge \brak{\Hform +k \omega_M}^n}{\omega_{FS}^{r-1}\wedge
\brak{\Hform +k \omega_M}^n} +\mathcal{O}(k^{-1}) \\
= & \brak{\Vlap^2-Scal(\omega_{FS}) \Vlap}\phi+(r-1)(r-2)\frac{\brak{i \dbar
\partial \phi}_{\VerT} \wedge \brak{\rho_k}_{\VerT}
\wedge \omega_{FS}^{r-3} }
{\omega_{FS}^{r-1}} +\mathcal{O}(k^{-1}) \\
= & L_{Scal, F} +\mathcal{O}(k^{-1}).
\end{align*}
(Essentially, this computation is the same as the one in the proof of Lemma \ref{lem:VHtrace}.)
\end{proof}

From equation \eqref{eq:LinScalNeu}, we know that
since the Fubini-Study metrics induced on the fibres of $\PE \to M$ have constant scalar curvature,
\[L_{Scal, F} (\phi)= \LichOp_F (\phi),\]
where $\LichOp_F$ is the Lichnerowicz operator on the fibres.

\begin{remark}\label{rem:LichOperator} \
\begin{enumerate}
\item
On a K\"ahler manifold $(M,J,g,\omega)$ endowed with a $G_{max}$-invariant K\"ahler metric---where again $G_{max}$ is some chosen maximal connected compact subgroup of $Aut^0_{red}(M,J)$---the Lichnerowicz operator $\LichOp$ is a self-adjoint, fourth order linear elliptic differential operator which is moreover $G_{max}$-invariant. Naturally, $\LichOp$ acts on the space of smooth, real-valued, $G_{max}$-invariant functions $C^{\infty}_{G_{max}}(M, \bR)$; and has a continuous linear extension---also denoted by $\LichOp$---mapping between the Sobolev-completions $L^2_{m, G_{max}}(M, \bR)$ of $C^{\infty}_{G_{max}}(M, \bR)$ in $L^2_{G_{max}}(M, \bR)$.
\item
Moreover on the K\"ahler manifold $(M,J,g,\omega)$, the space of Hamiltonian Killing vector fields $\mathfrak{ham}(M,J, \omega)=\mathfrak{iso}^0 (M,g) \cap \mathfrak{aut}^0_{red}(M,J)$---where $\mathfrak{iso}^0 (M,g)$ is the Lie algebra of the isometry group $Isom^0 (M,g)$ of $(M,g)$---can be identified via the Hamiltonian construction for $\omega$ with the (co-)kernel of $\LichOp$ in $C^{\infty}(M, \bR)$, since a vector field $V \in\mathfrak{ham}(M,J, \omega)$ if and only if it is of the form $V=J \grad_g f=\grad_{\omega} f$ for a real function $f \in \ker \LichOp$ (For a proof of this result, cf. \cite[Theorem~1~and~Proposition~1]{LS}).
\item
In particular, the Lichnerowicz operator $\LichOp$ on $\PE \to M$ is invariant under the (Hamiltonian) $\bT^s$-action induced on $\PE$ by the bundle-endomorphisms $Id_{E_1}, \dots, Id_{E_s}$---remember, the vector bundle $E \to M$ is supposed to split as a direct sum $E=E_1 \oplus \dots \oplus E_s$ of \emph{stable}, hence simple, sub-bundles of different slope---via the (infinitesimal) action described in Definition \ref{def:InfAction}. This is relevant, for example,
since we perturb the K\"ahler metric $\omega_k$ on $\PE \to M$ by adding $\bT^s$-invariant K\"ahler potentials $\phi \in C^{\infty}_{\bT^s} (\PE , \bR)$.
\end{enumerate}
\end{remark}

By point 2. of Remark \ref{rem:LichOperator} we know, since $\LichOp_F$ is self-adjoint, that $\ker \LichOp_F \cong \coker \LichOp_F$ can be identified via the Fubini-Study metric induced on the fibres of $\PE \to M$ with the function space $\Enu$ in the splitting \eqref{eq:FuncSpace Split} of $C^{\infty}(\PE)$.
Therefore, we can invert $L_{Scal, F} =\LichOp_F$ \emph{only} in the function space $R$---which consists of the functions which are $L^2$-orthogonal to $\Enu$ and the constant functions.

The $R$-component of $\eta_{\O(k^{-2})}$ will be corrected by adding a suitably chosen K\"ahler potential $k^{-2} \phi_{\mathcal{O}(k^{-2}), R}$ to $\omega_k$. Applying Lemma \ref{1storderLin} gives
\begin{equation}\label{1stRcorrect}
Scal(\omega_k + k^{-1} \idd \phi_{\mathcal{O}(k^{-2}),\Enu} +k^{-2} \idd \phi_{\mathcal{O}(k^{-2}), R})= 
\end{equation}
\[
\oC+k^{-1} \left( Scal(\omega_M)
 + \mu^* \left( h', i \sum_{p=1}^s (-b \lambda_p + k^{-1}  \gamma_p) Id_{E_p} \right)  \right) 
\]
\[
+k^{-2}
\left( L_{Scal, F} (\phi_{\mathcal{O}(k^{-2}), R}) +  \eta_{\O (k^{-2}), R} + \eta_{\O (k^{-2}), C^{\infty}(M)} \right)+ \mathcal{O}(k^{-3}).
\]
Therefore, the $\eta_R$-part of the $\O (k^{-2})$-error can be corrected by solving
\begin{equation}\label{eq:FibreScalCorrkminus2}
L_{Scal, F} (\phi_{\mathcal{O}(k^{-2}), R}) = - \eta_{\O (k^{-2}), R}, 
\end{equation}
for the K\"ahler potential  $\phi_{\mathcal{O}(k^{-2}), R}$. Indeed, $\phi_{\mathcal{O}(k^{-2}), R}$ can be chosen to be invariant under the $\bT^s$-action induced on $\PE \to M$, since the differential operator $L_{Scal, F}= \LichOp_F$ itself is invariant under this action (See point 3. of Remark \ref{rem:LichOperator}).
\begin{lemma}\label{1stRsol} For $\theta \in R$, there exists a unique
$\rho \in R$ such that
\[L_{Scal, F} (\rho)=\theta. \]
\end{lemma}
\begin{proof} (Modified from the analogous result for Kodaira fibrations, \cite[Lemma~3.6]{F}.) \\
Given the function $\rho \in R$, denote by $\rho_{\sigma}$
the restriction of $\rho$ to the fibre of $\PE \to M$ over $\sigma \in M$. The 
operator $L_{Scal, F}$ is just the linearisation of the scalar curvature map on K\"ahler potentials determined by the induced Fubini-Study metric on that fibre. 
By point 2. of Remark \ref{rem:LichOperator}, this operator is linear elliptic, self-adjoint and also
an isomorphism for functions in $R$. 
Since functions in $R$ are ($L^2$-)orthogonal to $\Enu$
and also to the constant functions, we can certainly solve the fibrewise
equation $(L_{Scal, F})_{\sigma}
\rho_{\sigma} = \theta_{\sigma}$, uniquely. 
Patching together, using the uniqueness
of the fibrewise solutions
$\rho_{\sigma}$, gives a solution to $L_{Scal, F}
(\rho)=\theta$. Because the operator $L_{Scal, F}$ is only elliptic in the vertical
directions, we have to check that
the function $\rho$ is also smooth transverse to the fibres. However, since
$\rho_{\sigma} = (L_{Scal, F})_{\sigma}^{-1} \theta_{\sigma}$, and the fact that
$(L_{Scal, F})_{\sigma}$
is a smooth family of differential operators, the required regularity properties
follow.
\end{proof}
Applying Lemma \ref{1stRsol} and using point 3. of Remark \ref{rem:LichOperator} gives us the existence of a $\bT^s$-invariant solution $\phi_{\mathcal{O}(k^{-2}), R} \in R \cap C^{\infty}_{\bT^s} (\PE , \bR)$ of equation \eqref{eq:FibreScalCorrkminus2}.

Adding the $\bT^s$-invariant potential $\idd k^{-2} \phi_{\mathcal{O}(k^{-2}), R}$ with $L_{Scal, F} (\phi_{\mathcal{O}(k^{-2}), R}) = - \eta_{\O (k^{-2}), R}$ to $\omega_k$ can be considered as changing 
\[ 
\omega_{FS} \mapsto \omega_{FS}+ k^{-2} \idd \phi_{\mathcal{O}(k^{-2}), R}. 
\]
So for the term
\[
\mu^* \left( h', i \sum_{p=1}^s (-b \lambda_p + k^{-1}  \gamma_p) Id_{E_p} \right)
\]
in equation \eqref{1stRcorrect} to \emph{remain} a Hamiltonian for the vector field 
\[
\widehat{i \sum_{p=1}^s (-b \lambda_p + k^{-1}  \gamma_p) Id_{E_p}}
\]
(again defined as in Definition \ref{def:InfAction}) with respect to the perturbed metric $\omega_{FS}+ k^{-2} \idd \phi_{\mathcal{O}(k^{-2}), R}$, it will change according to the following Lemma.
\begin{lemma}\label{lem:HamilVar}
Given a Hamiltonian $F$ for some vector field $V$ in the Lie-algebra $\mathfrak{ham}(M,J, \omega)$
of the Hamiltonian isometry group $Ham(M,J, \omega)$ of a (compact) K\"ahler manifold $(M,J,\omega,g)$.
Varying $\omega$ by adding a $V$-invariant K\"ahler potential $\psi \in C^{\infty}(M)$, i.e. $\mathcal{L}_V \psi=0$, such that $\omega'=\omega+ \idd \psi$,
varies $F$ according to the rule
\begin{equation}\label{eq:HamilVar}
F'=F+ \frac{1}{2}d \psi (JV)=F - \frac{1}{2} \nabla F \cdot \nabla \psi,
\end{equation}
up to the addition of a constant. The gradient and inner product are both taken
with respect to the metric $g$ corresponding to $\omega$.
\end{lemma}
\begin{proof}
The vector field $V$ and the Hamiltonian $F$ are related via
\[\iota_V \omega = -dF.\]
If $\omega'=\omega+ \idd \psi=\omega - \frac{1}{2} d d^c \psi$, then
\begin{align*}
\iota_V \omega' & = \iota_V \omega - \frac{1}{2} \iota_V d d^c \psi\\
& = -dF -\frac{1}{2} \mathcal{L}_V (d^c \psi)+\frac{1}{2} d(\iota_V d^c \psi) \ \text{(by Cartan's formula)}\\
& = -dF -\frac{1}{2} d^c (\mathcal{L}_V \psi)+\frac{1}{2} d(\iota_V d^c \psi) \ \text{(as $V$ is real holomorphic)}\\
& = -d(F-\frac{1}{2} \iota_V d^c \psi) \ \text{(since $\mathcal{L}_V \psi=0$)}\\
& = -d(F+\frac{1}{2}d \psi (JV)).
\end{align*}
This computation shows that $V$ is a Hamiltonian vector field with respect to $\omega'$,
and the corresponding Hamiltonian function is $F'=F+ \frac{1}{2}d \psi (JV)=F - \frac{1}{2} \nabla F \cdot \nabla \psi$.
\end{proof}
Applying Lemma \ref{lem:HamilVar} to $\mu^* \left( h', i \sum_{p=1}^s (-b \lambda_p + k^{-1}  \gamma_p) Id_{E_p} \right)$ and $\omega_{FS}+ k^{-2} \idd \phi_{\mathcal{O}(k^{-2}), R}$ shows that $\mu^* \left( h', i \sum_{p=1}^s (-b \lambda_p + k^{-1}  \gamma_p) Id_{E_p} \right)$ transforms via
\begin{align*}
\mu^* \left( h', i \sum_{p=1}^s (-b \lambda_p + k^{-1}  \gamma_p) Id_{E_p} \right) \mapsto & \mu^* \left( h', i \sum_{p=1}^s (-b \lambda_p + k^{-1}  \gamma_p) Id_{E_p} \right) \\
& - \frac{k^{-2}}{2}
\nabla \phi_R \cdot \nabla \mu^* \left( h', i \sum_{p=1}^s (-b \lambda_p + k^{-1}  \gamma_p) Id_{E_p} \right).
\end{align*}
This completes the task of correcting the $R$-component $\eta_{\O (k^{-2}), R}$ of the $\mathcal{O} (k^{-2})$-error $\eta_{\O (k^{-2})}$.

\subsubsection{Correcting the $C^{\infty}(M)$-part}\label{subsec:Cpart}

In order to correct the $C^{\infty}(M)$-component $\eta_{\O (k^{-2}), C^{\infty}(M)}$ of the $\mathcal{O} (k^{-2})$-error $\eta_{\O (k^{-2})}$,
we will perturb the metric $\omega_M$, pulled back from the base, with a K\"ahler potential $\phi_{\O (k^{-2}), C^{\infty}(M)} \in C^{\infty}(M)$ in a suitable way.

\

From equation \eqref{1stLin} we know that the scalar curvature $Scal(\omega_M)$
of $\omega_M$ (the pulled back metric from the base)
appears at order $\mathcal{O}(k^{-1})$ in $Scal(\omega_k)$---it is the $C^{\infty}(M)$-part
of the $\mathcal{O}(k^{-1})$-term of $Scal(\omega_k)$. 
Using the K\"ahler potential $\phi_{\O (k^{-2}), C^{\infty}(M)} \in C^{\infty}(M)$ to perturb $\omega_k$ can
be thought of as changing the metric $\omega_M$, scaled by the factor of $k$ in the definition
of $\omega_k$. Because of this
scaling, the effect of adding $\idd \phi_{\O (k^{-2}), C^{\infty}(M)}$ to $\omega_k$ is the
same as adding $k^{-1} \idd \phi_{\O (k^{-2}), C^{\infty}(M)}$ to $\omega_M$.

With the following \emph{formal} linearisation formula---derived exactly the same way
as equation~\eqref{eq:LinScalMap}---giving the variation of $Scal(\omega_M)$
\begin{equation}\label{eq:BaseLin}
Scal \brak{\omega_M+ k^{-1} \idd  \phi_{\O (k^{-2}), C^{\infty}(M)}}= Scal(\omega_M) +
k^{-1} L_{Scal,M} (\phi_{\O (k^{-2}), C^{\infty}(M)}) + \mathcal{O}(k^{-2}),
\end{equation}
(in which $L_{Scal,M}$ denotes the \emph{formal} linearisation of the scalar curvature map on
K\"ahler potentials on the base) we obtain by adding
$\idd \phi_{\O (k^{-2}), C^{\infty}(M)}$ to the perturbed metric \linebreak $\omega_k + k^{-1} \idd \phi_{\mathcal{O}(k^{-2}),\Enu} + k^{-2} \idd \phi_{\mathcal{O}(k^{-2}), R}$, \emph{considering
it as a change in} $\omega_M$, using equations~\eqref{1stRcorrect}
and \eqref{eq:BaseLin}
\[
Scal \brak{\omega_k + k^{-1} \idd \phi_{\mathcal{O}(k^{-2}),\Enu} +k^{-2} \idd \phi_{\mathcal{O}(k^{-2}), R} + \idd \phi_{\O (k^{-2}), C^{\infty}(M)}} 
\]
\[
= \oC+k^{-1} \left( Scal(\omega_M)
 + \mu^* \left( h', i \sum_{p=1}^s (-b \lambda_p + k^{-1}  \gamma_p) Id_{E_p} \right)  \right)
\]
\[
+ k^{-2} \left(  \eta_{\O (k^{-2}), C^{\infty}(M)} +L_{Scal,M} (\phi_{\O (k^{-2}), C^{\infty}(M)}) \right)
+ \mathcal{O}(k^{-3}).
\]

Since the base metric $\omega_M$ is cscK, using equation \eqref{eq:LinScalNeu} gives
\[L_{Scal,M} = \LichOp_M,\]
where $\LichOp_M$ is the (self-adjoint, fourth-order linear elliptic) 
Lichnerowicz operator on the base. Analogous to Lemma \ref{1stRsol} we have.
\begin{lemma}\label{lem:BasePotCorr}
For $\beta \in C^{\infty}_0 (M)$, there exists a unique $\alpha \in C^{\infty}_0 (M)$ such that
\[L_{Scal, M} (\alpha)=\beta. \]
\end{lemma}
\begin{proof}
The cscK base manifold $(M,\omega_M)$ 
is assumed to have no holomorphic automorphisms. By
the Matsushima-Lichnerowicz theorem, holomorphic automorphisms complexify
Hamiltonian isometries modulo trivial isometries on a cscK manifold; 
hence the base has
no non-trivial Hamiltonian isometries, and thus by point 2. of Remark \ref{rem:LichOperator},
$\ker L_{Scal, M} = \ker \LichOp_M \cong \bR$. Since $L_{Scal,M} = \LichOp_M$ is a self-adjoint, fourth-order linear elliptic differential operator on the compact manifold $(M,\omega_M)$, standard elliptic PDE-theory immediately gives the (unique) invertibility of $L_{Scal, M}: C^{\infty}_0 (M) \to C^{\infty}_0 (M)$ as $C^{\infty}_0 (M)=C^{\infty} (M) / \bR$ on the compact manifold $M$.
\end{proof}

Hence up to the addition of constants we can solve
\[L_{Scal,M} (\phi_{\O (k^{-2}), C^{\infty}(M)}) = - \eta_{\O (k^{-2}), C^{\infty}(M)},\]
in case the right hand side has mean-value zero.
Denoting by $c_2 := \overline{\eta_{\O (k^{-2}), C^{\infty}(M)}}$ the mean value of $\eta_{\O (k^{-2}), C^{\infty}(M)}$ (with respect to the K\"ahler metric corresponding to $\omega_M$), we define
$\phi_{\O (k^{-2}), C^{\infty}(M)}$ to be the solution of
\begin{equation}\label{eq:CinfKaehlerPotCorr}
L_{Scal,M} (\phi_{\O (k^{-2}), C^{\infty}(M)}) =\LichOp_M (\phi_{\O (k^{-2}), C^{\infty}(M)})= c_2 - \eta_{\O (k^{-2}), C^{\infty}(M)}.
\end{equation}
By Lemma \ref{lem:BasePotCorr} this equation can be solved in $\Cinf (M)$ since its right hand side has mean-value zero.
Moreover, since the K\"ahler potential $\phi_{\O (k^{-2}), C^{\infty}(M)}$ is pulled back from the base, it is automatically \emph{invariant} under the (Hamiltonian) $\bT^s$-action induced by $Id_{E_1},\dots,Id_{E_S} \in \End (E)$ on $\PE$.

Therefore, the $C^{\infty}(M)$-part $\eta_{\O (k^{-2}), C^{\infty}(M)}$ of the $\mathcal{O} (k^{-2})$-error
$\eta_{\O (k^{-2})}$ is corrected modulo the constant $c_2$, i.e.
\[
Scal \brak{\omega_k + k^{-1} \idd \phi_{\mathcal{O}(k^{-2}),\Enu} +k^{-2} \idd \phi_{\mathcal{O}(k^{-2}), R} + \idd \phi_{\O (k^{-2}), C^{\infty}(M)}} 
\]
\begin{equation}\label{eq:OkMinus2Corr}
= \oC+k^{-1} \left( Scal(\omega_M)
 + \mu^* \left( h', i \sum_{p=1}^s (-b \lambda_p + k^{-1}  \gamma_p) Id_{E_p} \right)  \right)
\end{equation}
\[
+ c_2 k^{-2} + \mathcal{O}(k^{-3}).
\]
Thus we completely corrected the $\mathcal{O} (k^{-2})$-error
$\eta_{\O (k^{-2})}$.

\subsection{The higher order approximate solutions}\label{sec:HighApproxSol}

In this section we will complete our approximation scheme.
This enables us to find---in the sense of equation \eqref{eq:ExtrEqApproxDef}---an approximate formal power series solution to the extremal metric equation~\eqref{eq:ExtrCond1}, pointwise arbitrarily close to a genuine solution.
\begin{remark}\label{rem:QNotation}
From now on, in order to save on notation, we will denote the Hamiltonian constructed while perturbing the map $\mu^*$ in our induction scheme by $Q$.
\end{remark}
\begin{theorem}[Formal solutions to the extremal metric equation]\label{ApproxThm}
Given an integer $n \geq 1$ we can find K\"ahler potentials, invariant under the $\bT^s$-action induced by $Id_{E_1},\dots,Id_{E_s} \in \End (E)$ on $\PE$,
\[
\phi_{i,\Enu} \in C^{\infty}_{\bT^s} (\PE, \bR), \ \phi_{i,R} \in R \cap C^{\infty}_{\bT^s} (\PE, \bR),
\ \phi_{i,C^{\infty} (M)} \in C^{\infty} (M) \cap C^{\infty}_{\bT^s} (\PE, \bR),
\] 
\[
\text{for} \ i=1, \dots,n
\]
such that the metric
\[
\omega_{k,n}=\omega_k+ \idd \sum_{i=1}^n \left( k^{-i} \phi_{i,\Enu} + k^{-i-1} \phi_{i,R}
+  k^{-i+1} \phi_{i,C^{\infty} (M)} \right)
\]
is an $(n+1)$-th order approximate solution to the extremal metric equation
\eqref{eq:ExtrCond1},
by which we mean, as in equation \eqref{eq:ExtrEqApproxDef}, that pointwise on $\PE$
\begin{equation}\label{eq:ExtrEqApproxSol}
Scal( \omega_{k,n} ) - Q (\omega_{k,n}) - \oC = \sum_{i=1}^{n+1} c_i k^{-i} + \mathcal{O}(k^{-n-2}),
\quad \oC,c_1, \dots, c_{n+1} \in \bR,
\end{equation}
where $Q(\omega_{k,n})$ is a Hamiltonian with respect to the purely vertical part $(\omega_{k,n})_{\VerT}$ of $\omega_{k,n}$ for a (Hamiltonian Killing) vector field in the Lie algebra $\mathfrak{t}^s$ of the torus $\bT^s$ (generated by the vector fields induced by $Id_{E_1}, \dots, Id_{E_s} \in \End (E)$ on $\PE$).
\end{theorem}
\begin{proof}
The proof follows by induction using the steps carried out in order to find the second order approximate solution in Sections \ref{subsec:Epart}, \ref{subsec:Rpart}, \ref{subsec:Cpart} as the inductive steps.
\end{proof}

\section{Analytic aspects}\label{Chap6}

The whole Section \ref{Chap6} bears many similarities with \cite[Sections~4--7]{F}, and
in fact many results and ideas of J. Fine were adapted for our case and are variations of his results. 

\subsection{The Implicit Function Theorem}\label{IFTsection}

We are going to use a parameter-dependent implicit function theorem (IFT), the parameter being the adiabatic parameter $k$,
in order to show the existence of a genuine solution of the extremal metric
equation, lying nearby the approximate solution found in Theorem \ref{ApproxThm}. 

\begin{theorem}[Implicit function theorem]\label{IFT}
\
\begin{itemize}
\item
Let $F:B_1 \to B_2$ be a differentiable map of Banach spaces, whose derivative at $0$, $DF|_0$, is an epimorphism of Banach spaces, with right-inverse $P$.
\item
Let $\delta'$ be the radius of the closed ball in $B_1$, centred at $0$, on which $F-DF|_0$ is Lipschitz, with constant $1/(2\|P\|)$.
\item
Let $\delta =\delta'/(2\|P\|)$. 
\end{itemize}
Whenever $y\in B_2$ satisfies $\|y-F(0)\| < \delta$, there exists $x \in B_1$ with $F(x)=y$. 
\end{theorem}

In fact, this statement of the IFT is the same as \cite[Theorem~4.1]{F}, except that we assume $DF|_0$ to be an \emph{epimorphism} of Banach spaces having only a \emph{right-inverse} $P$, instead of a ``two-sided'' inverse. The reason is that unlike Fine \cite{F} or Hong \cite{Ho1}, we actually have non-trivial Hamiltonian Killing vector fields on $\mathbb{P}(E)$, induced by the non-trivial automorphism group of the vector bundle $E \to M$ (remember: $E$ is \emph{unstable}, and \emph{not simple}
since $\Aut (E) \cong U(1)^s$). Therefore, the leading order part of the linearisation at $\omega_{k,n}$ of the ``approximate extremal metric operator\footnote[1]{We shall use ``AEMO'' as abbreviation for ``approximate extremal metric operator''.}'' 
\begin{equation}\label{eq:AEMODef}
\AEMO (\phi) :=Scal( \omega_{k,n} + \idd \phi) - Q (\omega_{k,n} + \idd \phi) - \oC
\end{equation}
(i.e. the left hand side of equation \eqref{eq:ExtrEqApproxSol}) on $\bT^s$-invariant K\"ahler potentials will have a non-trivial co-kernel in the function spaces $C^{\infty}_{\bT^s} (\mathbb{P}(E),\mathbb{R})$ and $L^2_{m, \bT^s} (\mathbb{P}(E),\mathbb{R})$; where we use the standard notation and denote by $L^2_m$ the Sobolev space of functions whose derivatives up to order $m$ are in $L^2$.

\begin{remark}\label{rem:SobolevkDep}
The Sobolev space $L^2_m (g_{k,n})$, defined via the metric $g_{k,n}$ corresponding to $\omega_{k,n}$,
contains the same functions for different values of the adiabatic parameter $k$, since the corresponding Sobolev norms $\| \cdot \|_{L^2_m (g_{k,n})}$ are equivalent for different values of $k$. (Although the constants of equivalence depend on $k$.)
\end{remark}

\paragraph{The parametrised equation}

\

\

Our goal is to solve the extremal metric equation \eqref{eq:ExtrCond1},
for $k \gg 0$ and \emph{fixed} $n$,
\begin{equation}\label{eq:ExtrEqKaehlerPot}
Scal(\omega_{k,n}+\idd \phi) - H(\omega_{k,n}+ \idd \phi) - \overline{S} = 0, 
\end{equation}
where $\oS$ is the average scalar curvature and $\phi$ is a $\bT^s$-invariant K\"ahler potential.
So it is reasonable to try to solve $\AEMO (\phi)=0$, with $\AEMO$ as in \eqref{eq:AEMODef}; which we want to do without having to worry about the obstructions coming from the non-trivial co-kernel
of the leading order part of its linearisation.
In order to handle this non-trivial co-kernel, we will employ essentially the same trick as already used in Section \ref{subsec:Epart} above. More precisely, denote the linearisation of $\AEMO$ at 
$\omega_{k,n}$ by $L_{{\rm AEMO},\omega_{k,n}}(\phi)$.
Using Lemma \ref{lem:LinScalNeu} to linearise the $Scal(\omega_{k,n})$-part in \eqref{eq:AEMODef}, and Lemma \ref{lem:HamilVar} to linearise the $Q(\omega_{k,n})$-part, on ($\bT^s$-invariant) K\"ahler potentials, we obtain
\begin{align}
L_{{\rm AEMO},\omega_{k,n}}(\phi)= & \LichOp (\phi) + \frac{1}{2} \nabla Scal(\omega_{k,n}) \cdot
\nabla \phi- \frac{1}{2} \nabla Q(\omega_{k,n}) \cdot \nabla \phi \nonumber \\
= & \LichOp (\phi) + \O (k^{-n-2}) \quad \text{(by using equation \eqref{eq:ExtrEqApproxSol})}, \label{eq:LAEMOLichOp}
\end{align}
where the gradient and inner product in the first line, and $\LichOp$ in both lines, are taken with respect to the metric corresponding to $\omega_{k,n}$.

The (co-)kernel of the self-adjoint operator $\LichOp$ in $C^{\infty}_{\bT^s} (\mathbb{P}(E),\mathbb{R})$ is isomorphic, via $\omega_{k,n}$, to the space of Hamiltonian Killing vector fields induced---as in Definition \ref{def:InfAction}---on $\PE$ by linear combinations of $Id_{E_1}, \dots, Id_{E_s} \in \End (E)$. Also, $Q(\omega_{k,n})$ in equation \eqref{eq:ExtrEqApproxSol}
is the Hamiltonian, with respect to the purely vertical part $(\omega_{k,n})_{\VerT}$ of $\omega_{k,n}$, for the vector field
\begin{equation}\label{eq:VFieldHam}
B:= \widehat{i \sum_{p=1}^s \left( -b k^{-1} \lambda_p + \sum_{l=1}^n k^{-l-1}  \gamma_{p,l} \right) Id_{E_p}}, \quad \lambda_p, \gamma_{p,l} \in \bR \ \text{for} \ p=1, \dots,s, \ l=1, \dots, n;
\end{equation}
constructed by iterating the procedure in Section \ref{subsec:Epart} in order to find $\omega_{k,n}$. (The additional factor of $k^{-1}$ in \eqref{eq:VFieldHam} is due to $Q(\omega_{k,n})$ \emph{not being multiplied} by $k^{-1}$ in equation \eqref{eq:ExtrEqApproxSol}; in contrast to $\mu^* (\cdots)$ in equation \eqref{eq:OkMinus2Corr}.) Therefore, we introduce an $s$-tuple of parameters $\Theta := (\theta_1, \dots \theta_s)$ with $\theta_1, \dots, \theta_s \in \bR$, in the vector field $B$ in \eqref{eq:VFieldHam} and define
\begin{equation}\label{eq:VFieldHamParam}
B':=B+B_{\Theta}:= \widehat{i \sum_{p=1}^s \left( -b k^{-1} \lambda_p + \sum_{l=1}^n k^{-l-1}  \gamma_{p,l} \right) (1+\theta_p) Id_{E_p}},
\end{equation}
with
\begin{equation}\label{eq:BVfieldTheta}
B_{\Theta}:=\widehat{i \sum_{p=1}^s \left( -b k^{-1} \lambda_p + \sum_{l=1}^n k^{-l-1}  \gamma_{p,l} \right) \theta_p Id_{E_p}};
\end{equation}
which can be interpreted as varying the (infinitesimal) action of $B$ on $\PE$. The parameter-dependent vector fields $B', B_{\Theta}$ are again Hamiltonian Killing vector fields on $\PE$ with $B', B_{\Theta} \in \mathfrak{t}^s$---the Lie algebra of $\bT^s$ (which is generated by the vector fields induced by $Id_{E_1}, \dots, Id_{E_s} \in \End (E)$ on $\PE$). Of course, the introduction of the $s$-tuple of parameters $\Theta$ makes the Hamiltonian for $B$ with respect to $(\omega_{k,n})_{\VerT}$---which we denote by $Q(\omega_{k,n},B)$---parameter-dependent, as well. Thus,
\[
Q(\omega_{k,n},B')=Q(\omega_{k,n}, B+ B_{\Theta})=Q(\omega_{k,n},B)+Q(\omega_{k,n},B_{\Theta}),
\]
since $Q$ is linear in the second argument
So, instead of solving $\AEMO (\phi)=0$ directly for $\phi \in C^{\infty}_{\bT^s} (\mathbb{P}(E),\mathbb{R})$, we will solve a ``parametrised version''. Therefore, we shall also consider the constant $\oC \in \bR$ in \eqref{eq:AEMODef} as a parameter, which we write as $\oC+\oR$, and solve
\begin{equation}\label{eq:ExtrEqKaehlerPotPara}
Scal(\omega_{k,n}+\idd \phi) - Q(\omega_{k,n}+ \idd \phi, B) - Q(\omega_{k,n},B_{\Theta}) - \oC - \oR= 0, 
\quad \Theta \in \bR^s, \ \oR \in \bR,
\end{equation}
for $\phi \in C^{\infty}_{\bT^s} (\mathbb{P}(E),\mathbb{R})$ \emph{and} $\Theta \in \bR^s, \ \oR \in \bR$.
We define the corresponding ``parametrised extremal metric operator'' to be
\begin{equation}\label{eq:ParaExtrMetOPDef}
\AEMO^{\Theta, \oR} (\phi) := Scal(\omega_{k,n}+\idd \phi) - Q(\omega_{k,n}+ \idd \phi, B) - Q(\omega_{k,n},B_{\Theta}) - \oC - \oR;
\end{equation}
and will denote its linearisation at $\omega_{k,n}$ by 
$L_{{\rm AEMO},\omega_{k,n}}^{\Theta, \oR}$.
Hence we get, as the operator is linear in the parameters $(\Theta,\oR)$,
\begin{align}
L_{{\rm AEMO},\omega_{k,n}}^{\Theta, \oR}(\phi)= & \LichOp (\phi) + \frac{1}{2} \nabla Scal(\omega_{k,n}) \cdot
\nabla \phi- \frac{1}{2} \nabla Q(\omega_{k,n},B) \cdot \nabla \phi - Q(\omega_{k,n},B_{\Theta}) - \oR \nonumber \\
= & \LichOp (\phi) - Q(\omega_{k,n},B_{\Theta}) - \oR  + \O (k^{-n-2}).  \label{eq:LAEMOLichOpPara}
\end{align}
\begin{lemma}\label{lem:ParamLin}
For the linearisation of $\AEMO^{\Theta, \oR} (\phi)$ (defined in \eqref{eq:ParaExtrMetOPDef}) at $\omega_{k,n}$ we get
\begin{equation}\label{eq:ParamLin}
L_{{\rm AEMO},\omega_{k,n}}^{\Theta, \oR}(\phi) = \LichOp (\phi) - Q(\omega_{k,n},B_{\Theta}) - \oR  + \O (k^{-n-2}).
\end{equation}
\end{lemma}
Later, we will show that the map
\begin{equation}\label{eq:ParamLin}
L_{{\rm AEMO},\omega_{k,n}}^{\Theta, \oR} : L^2_{m+4, \bT^s} \times \bR^{s+1} \to L^2_{m, \bT^s},
\end{equation}
is a Banach space epimorphism for which we can 
construct a right-inverse with suitable bounds.

\begin{remark}\label{rem:ParamChoice}
Because there is only \emph{one} $\bT^s$-action on $\PE$, we know by the theory outlined in Section
\ref{subsec:EKMBackground} that the $s$-tuple of parameters $\Theta$ is determined by
the K\"ahler-class $[\omega_k]$ and the $\bT^s$-action. In particular the extremal vector field---defined in Definition \ref{def:EVF}---is determined by this data, and the variation of $\Theta$ will perturb $Q(\omega_{k,n},B)+Q(\omega_{k,n},B_{\Theta})$ to the Hamiltonian
of the extremal vector field, as we apply the IFT.

If equation \eqref{eq:ExtrEqKaehlerPotPara} is satisfied for the $\bT^s$-invariant K\"ahler metric $\omega_{k,n}+ \idd \phi, \Theta, \oR$, it follows from the calculation in equation \eqref{eq:HamFutakiCalc} that $\oC+\oR$ is the average of $Scal(\omega_{k,n}+ \idd \phi)$
and $Q(\omega_{k,n} + \idd \phi,B)+Q(\omega_{k,n},B_{\Theta})$ is the (mean-value zero) Hamiltonian of the extremal vector field.
\end{remark}

\paragraph{Applying the parameter-dependent implicit function theorem}

\

\

Once we showed that the map $L_{{\rm AEMO},\omega_{k,n}}^{\Theta, \oR} : L^2_{m+4, \bT^s} \times \bR^{s+1} \to L^2_{m, \bT^s}$ is a Banach space epimorphism with bounded right-inverse, applying the implicit function Theorem \ref{IFT} to the map
\[
L^2_{m+4, \bT^s} \times \bR^{s+1} \ni (\phi,\Theta, \oR) \mapsto Scal(\omega_{k,n}+\idd \phi) - Q(\omega_{k,n}+ \idd \phi, B) - Q(\omega_{k,n},B_{\Theta}) - \oC - \oR \in L^2_{m, \bT^s};
\]
we see that if the evaluation of this map at $(0,0,0) \in L^2_{m+4, \bT^s} \times \bR^{s+1}$ for small $\delta_k$ satisfies
\[
\left\| Scal(\omega_{k,n}) - Q(\omega_{k,n}, B)  - \oC \right\|_{L^2_{m, \bT^s} (g_{k,n})}< \delta_k,
\]
then there exist $(\phi,\Theta, \oR) \in L^2_{m+4, \bT^s} \times \bR^{s+1}$ satisfying equation \eqref{eq:ExtrEqKaehlerPotPara}. Hence we can conclude the proof once we have shown that $Scal(\omega_{k,n}) - Q(\omega_{k,n}, B)  - \oC$ converges to zero quicker than $\delta_k$, for suitably chosen $n$.

\subsection{Local analysis}\label{LA}

In this section we will establish Sobolev inequalities, and elliptic estimates for $L_{{\rm AEMO},\omega_{k,n}}^{\Theta, \oR}$, \emph{uniformly} in the adiabatic parameter $k$.
Most results in this section were already proven in \cite[Section~5]{F}, to which we will often refer.

\subsubsection{The local model}

The most important result of this subsection, Proposition \ref{adiabThm}, states that the geometry of the fibres dominates the local geometry of the total space $\PE$ in an adiabatic limit for $k \gg 0$. The local model we use in this section was first constructed in \cite[Section~5.1]{F}, and our construction is an adaptation of it.

Let $B_{{\rm flat}} \subset M$ be a ball in the base manifold $M$ with centre $p_0 \in M$, endowed with the flat K\"ahler metric.
Since this ball is contractible, $\PE|_{B_{{\rm flat}}}$ (the part of $\PE$ over $B_{{\rm flat}}$) is diffeomorphic to $\P^{r-1} \times B_{{\rm flat}}$. The identification $\PE|_{B_{{\rm flat}}} \cong \P^{r-1} \times B_{{\rm flat}}$ can be arranged, such that the horizontal distribution on the (central) fibre $\mathbb{P}^{r-1}_{(p_0)}$
coincides with the restriction of the $T B_{{\rm flat}}$-summand to $\mathbb{P}^{r-1}_{(p_0)}$ in the splitting 
\begin{equation}\label{TPEsplitting}
T \left( \PE |_{B_{{\rm flat}}} \right) \cong T(\mathbb{P}^{r-1} \times B_{{\rm flat}}) \cong T\mathbb{P}^{r-1} \oplus T B_{{\rm flat}}.
\end{equation}
For every $k$, two K\"ahler structures on $\mathbb{P}^{r-1} \times B_{{\rm flat}}$ will be
of interest: the first one is simply the restriction of the K\"ahler structure
$(\PE,J,\omega_{k,n})$ to $\PE|_{B_{{\rm flat}}}$. 

The second K\"ahler structure of interest is the product structure $(J',\omega'_k)$, 
scaled by $k$ in the $B_{{\rm flat}}$-direction.
With respect to the splitting \nolinebreak \eqref{TPEsplitting}, we have
\begin{align*}
J' & = J_{\mathbb{P}^{r-1}} \oplus J_{B_{{\rm flat}}},\\
\omega'_k & = \omega_{\mathbb{P}^{r-1}} \oplus k \omega_{B_{{\rm flat}}},
\end{align*}
where $\omega_{B_{{\rm flat}}}$ is the flat K\"ahler form on $B_{{\rm flat}}$ agreeing with $\omega_M$
at $p_0 \in M$, $J_{B_{{\rm flat}}}$ is the complex structure on $B_{{\rm flat}}$, and $J_{\mathbb{P}^{r-1}},\omega_{\mathbb{P}^{r-1}}$ are the complex
structure and (Fubini-Study) K\"ahler form on the (central) fibre $\mathbb{P}^{r-1}_{(p_0)}$.
The corresponding product metric induced by $(J',\omega'_k)$ on $\mathbb{P}^{r-1} \times B_{{\rm flat}}$ will be denoted by $g'_k$. Observe, since the fibration $\PE \to M$ is locally holomorphically trivial, the complex structure $J$ induced on $\PE|_{B_{{\rm flat}}}$ by restricting the K\"ahler structure $(\PE,J,\omega_{k,n})$ to $\PE|_{B_{{\rm flat}}}$,
and the complex structure $J'$ induced by the product K\"ahler structure coincide over $B_{{\rm flat}}$, 
i.e. $J'|_{B_{{\rm flat}}}=J|_{B_{{\rm flat}}}$.

Later on, we will need the following lemma. 
\begin{lemma}[cf. Lemma 5.1 in \cite{F}]\label{lem:scaling1}
Let $\alpha \in C^m(T^* \PE^{\otimes i})$. Over $\PE|_{B_{{\rm flat}}}$, $\| \alpha \|_{C^m
(g'_k)}= \mathcal{O}(1)$. If $\alpha$ is pulled up from the base, we have
$\| \alpha \|_{C^m (g'_k)}= \mathcal{O} \left( k^{-i/2} \right)$.
\end{lemma}
The proof of the lemma is the same as the one of \cite[Lemma~5.1]{F}. 
The main result of this subsection, which is the analogue of \cite[Theorem~5.2]{F}, is
\begin{proposition}\label{adiabThm}
For all $\varepsilon >0,\,p_0 \in M$, there exists a ball
$B_{{\rm flat}} \subset M$, centred at $p_0$, such that for all sufficiently
large $k$, over $\PE|_{B_{{\rm flat}}}$ we have
\[\| (J',\omega'_k)-(J,\omega_{k,n}) \|_{C^m(g'_k)} < \varepsilon.\]
\end{proposition}

The proof of the proposition is similar to the proof of \cite[Theorem~5.2]{F}, and we refer to this reference for the details; in fact, the proof in our case is easier since we just have to deal with a holomorphically trivial fibration $\PE \to M$, so $J'|_{B_{{\rm flat}}}=J|_{B_{{\rm flat}}}$, whereas \cite{F} considers Kodaira fibrations, the fibres of which have non-trivial moduli.

\subsubsection{Analysis in the local model}\label{subsec:AnaLocModel}

This section contains analytic results on the product model $(\mathbb{P}^{r-1}
\times B_{{\rm flat}},J',\omega'_k)$, needed in the sequel.

The proofs of the following results won't be reproduced, since they can
be taken over (almost) verbatim from the book \cite[Chapter~3]{D4}, or from \cite[Section~5.2]{F}.
\begin{lemma}[cf. Lemma 5.3 in \cite{F}]\label{SobolInequ}
For indices $m,l$ and $q \geq p$ satisfying $m -\dim_{\mathbb{R}} (\PE)/p
\geq l- \dim_{\mathbb{R}} (\PE)/q$, there is a constant $c$ (depending \emph{only}
on $m,l,q$ and $p$) such that for all $\phi \in L^p_{m,\bT^s} \left( \mathbb{P}^{r-1}
\times B_{{\rm flat}} \right)$,
\[ 
\| \phi \|_{L^q_{l,\bT^s} (g'_k)} \leq c \| \phi \|_{L^p_{m,\bT^s} (g'_k)}.
\]
For indices $p,m$ satisfying $m-\dim_{\mathbb{R}} (\PE)/p>0$, there exists
a constant $c$ (depending \emph{only} on $p,m$), such that for all $\phi
\in L^p_{m,\bT^s} \left( \mathbb{P}^{r-1} \times B_{{\rm flat}} \right)$,
\[ 
\| \phi \|_{C^0_{\bT^s}} \leq c \| \phi \|_{L^p_{m,\bT^s} (g'_k)},
\]
where $g'_k$ is the scaled product metric, from Proposition \ref{adiabThm}, on $\mathbb{P}^{r-1}
\times B_{{\rm flat}}$.
\end{lemma}
\begin{remark}\label{rem:TsInvarFunc}
Even though the result above, and also several results below (Lemmas (\ref{EllipEst}--\ref{lem:UnifEllipEst})), are proven in \cite{D4,F} for general Sobolev spaces, restricting to $\bT^s$-invariant functions in these spaces doesn't cause problems and the proofs are the same.
\end{remark}
The product K\"ahler structure $(J',\omega'_k)$ on $\mathbb{P}^{r-1}
\times B_{{\rm flat}}$, as defined in Proposition \ref{adiabThm}, determines a ``\emph{product extremal metric operator}'' on ($\bT^s$-invariant) K\"ahler potentials
\begin{equation}\label{eq:ProdMap}
\phi \mapsto Scal(\omega'_k +i \dbar \partial \phi) - Q(\omega'_k +i \dbar
\partial \phi,B) - \oC,
\end{equation}
with the Hamiltonian
\[
Q(\omega'_k +i \dbar \partial \phi,B),
\]
which is the analogue of $Q(\omega_{k,n}+i \dbar \partial \phi,B)$ for the product structure $(J',\omega'_k)$; i.e. $Q(\omega'_k,B)$ is the Hamiltonian for $B$ with respect to $\omega_{\mathbb{P}^{r-1}}$---the metric on the first factor of the product (where the vector field $B$ is induced on $\mathbb{P}^{r-1}$ as in Definition \ref{def:InfAction}).
We denote the linearisation of the map \eqref{eq:ProdMap} at $\omega'_k$
by $L_{{\rm AEMO},\omega'_k}(\phi): L^2_{m+4,\bT^s} (g'_k) \to 
L^2_{m,\bT^s} (g'_k)$.
Using the results from Chapter 3 of \cite{D4}, or \cite[Section~5.2]{F}, gives the following
elliptic estimate for $L_{{\rm AEMO},\omega'_k}(\phi)$. 
(Indeed, the estimates presented in Chapter 3
of \cite{D4} are valid for any elliptic operator determined by the local
geometry of the underlying manifold.)
\begin{lemma}[cf. Lemma 5.4 in \cite{F}]\label{EllipEst}
There exists a constant $C$ such that for all $\phi \in L^2_{m+4,\bT^s}
(\mathbb{P}^{r-1} \times B_{{\rm flat}})$,
\[ 
\| \phi \|_{L^2_{m+4,\bT^s} (g'_k)} \leq C \left( \| \phi  \|_{L^2_{\bT^s} (g'_k)}+
\| L_{{\rm AEMO},\omega'_k}(\phi) \|_{L^2_{m,\bT^s} (g'_k)}
 \right).
\]
\end{lemma}
Later on when carrying out the patching arguments to transform those results
from the product to the total space of $\PE \to M$, we will also need
\begin{lemma}[cf. Lemma 5.5 in \cite{F}]\label{lem:patch}
There exists a constant $P$, such that for all compactly supported
$u \in C^{m+4}_{\bT^s} (B_{{\rm flat}})$, and all $\phi \in L^p_{m+4,\bT^s}
\left( \mathbb{P}^{r-1} \times B_{{\rm flat}} \right)$,
\[
\| L_{{\rm AEMO},\omega'_{k}}(u\phi)-
u \, L_{{\rm AEMO},\omega'_{k}}(\phi)\|_{L^p_{m,\bT^s} (g'_{k})} \leq P \sum_{j=1}^{m+4}
\|\nabla^j u\|_{C^0_{\bT^s} (g'_{k})} \| \phi \|_{L^p_{m+4,\bT^s} (g'_{k})}.
\]
\end{lemma}

\subsubsection{Local analysis and patching arguments}

This section will show how to convert results from the product 
$(\mathbb{P}^{r-1} \times B_{{\rm flat}},J',\omega'_{k})$ to uniform results
over $(\PE,J,\omega_{k,n})$, and corresponds to \cite[Section~5.3]{F}. Applying Proposition \ref{adiabThm}
with $\varepsilon < 1$, we obtain that over $\PE|_{B_{{\rm flat}}}$, the difference $g_k-g'_k$ is uniformly
bounded in the space $C^m(g'_k)$. This choice of $\varepsilon$ ensures that the metrics
are sufficiently close, so that the difference $g^{T^* \PE}-g'^{T^* \PE}$ of
the induced metrics on the cotangent bundle is also uniformly bounded.

Hence the Banach space norms on tensors determined by $g_k$ and $g'_k$ are
uniformly equivalent, i.e.
\[l \| t \|_{C^m(g_k)} \leq \| t \|_{C^m(g'_k)} \leq L \| t \|_{C^m(g_k)},\]
for some tensor $t$ and fixed, positive constants $l,L$. From this we
get
\begin{lemma}[cf. Lemma 5.6 in \cite{F}]\label{lem:tensorpull}
For a tensor $\alpha \in C^m (T^* \PE^{\otimes i}), \, 
\| \alpha \|_{C^m(g_k)}=\mathcal{O}(1)$. Additionally,
if $\alpha$ is pulled up from the base, we have $\| \alpha \|_{C^m(g_k)}
=\mathcal{O}(k^{-i/2})$.
\end{lemma}
\begin{proof} 
The same argument as in the proof of \cite[Lemma~5.6]{F} applies, which we shall repeat for the reader's convenience.
By Lemma \ref{lem:scaling1}, the statement is true for the local product model. 
Let again $B_{{\rm flat}} \subset M$ be a ball over which Proposition \ref{adiabThm}
holds with $\varepsilon = 1/2$, for example. 

Since the two norms $\| \cdot \|_{C^m(g_k)}$ and $\| \cdot \|_{C^m(g'_k)}$
are uniformly equivalent over $\PE|_{B_{{\rm flat}}}$,
the result holds in the function space $C^m(g_k)$ over $\PE|_{B_{{\rm flat}}}$. Cover
$M$ with finitely many such balls $B_{{\rm flat},i}$. The result holds in $C^m(g_k)$ 
over each $\PE|_{B_{{\rm flat},i}}$ and so over all of $\PE$ by adding.
\end{proof}
The next lemma gives us a \emph{convergence} result in the function
spaces $C^m_{\bT^s} (g_k), L^2_{\bT^s} (g_k)$ needed later in order to apply the implicit function theorem. 
N.B.: up to now we only established \emph{pointwise convergence}
for the formal solution constructed in Section \ref{ChapApproxSol}.

\begin{lemma}[cf. Lemma 5.7 in \cite{F}]\label{lem:UnifConv}
We have
\begin{eqnarray*}
Scal(\omega_{k,n}) - Q (\omega_{k,n}) - \oC & = & \mathcal{O}(k^{-n-2})
\ {\rm in} \ C^m_{\bT^s}(g_k) \ {\rm as} \ k \to \infty, \\
Scal(\omega_{k,n}) - Q (\omega_{k,n}) - \oC & = & \mathcal{O}(k^{-n-2+(\dim
M)/2})
\ {\rm in} \ L^2_{m,\bT^s}(g_k) \ {\rm as} \ k \to \infty.
\end{eqnarray*}
\end{lemma}
\begin{proof}
The proof given here is similar to the proof of \cite[Lemma~5.7]{F}; in fact the proof of convergence in $C^m_{\bT^s}(g_k)$ is more or less the same as the one given there, adapted for our purposes. The proof of $L^2_{m,\bT^s}(g_k)$-convergence is different and brings in dimension considerations.

Since we established $Scal(\omega_{k,n}) - Q (\omega_{k,n}) - \oC=\mathcal{O}(k^{-n-2})$
pointwise in Theorem \ref{ApproxThm}, we shall first deduce that with respect to
some \emph{fixed metric} $g$, we have
\[
Scal(\omega_{k,n}) - Q (\omega_{k,n}) - \oC=\mathcal{O}(k^{-n-2})
\ {\rm in} \ C^m_{\bT^s} (g) \ {\rm as} \ k \to \infty. 
\]
In order to see this, we argue as follows.

All the calculations done in Section \ref{ChapApproxSol} involve absolutely
convergent power series and algebraic manipulations of them.

Concerning the $Q(\omega_{k,n})$-term, observe that the right hand side of equation
\eqref{1stLin} is obtained by manipulations such as: expansions of terms in (absolutely convergent) 
power series, involving negative powers of $k$; or the power-series-expansion of $\log(1+x)$.
I.e. concerning the computations done in the proof of Lemma \ref{1storder} we can argue that
for $\mu^* (\Lambda_{\omega_M} F^{\nabla})$, 
$\log (1+k^{-1} \mu^* (\Lambda_{\omega_M} F^{\nabla}))$ is $\mathcal{O}(k^{-1})$ in $C^m_{\bT^s} (g)$ since
\[
\| \log (1+k^{-1} \mu^* (\Lambda_{\omega_M} F^{\nabla})) \|_{C^m_{\bT^s}(g)} \leq \sum_{i \geq 1} \frac{k^{-(i+1)}
C^i \| \mu^* (\Lambda_{\omega_M} F^{\nabla}) \|^i_{C^m_{\bT^s}(g)}}{i}
\]
\[
=\log \left( 1+C k^{-1} \| \mu^* (\Lambda_{\omega_M} F^{\nabla}) \|_{C^m_{\bT^s}(g)} \right);
\]
with a constant $C$ such that $\| \rho \sigma\|_{C^m_{\bT^s}(g)} \leq C \|\rho\|_{C^m_{\bT^s}(g)}
\|\sigma\|_{C^m_{\bT^s}(g)}$. Since the Hamiltonian function $Q(\omega_{k,n})$ from Theorem \ref{ApproxThm}, constructed in our induction scheme, is an $\O (k^{-1})$-perturbation of the $\mu^* (\Lambda_{\omega_M} F^{\nabla})$-term from equation \eqref{1stLin}, it follows that it is also $\O (k^{-1})$ in $C^m_{\bT^s} (g)$.

Hence, for the statement to be true in the $C^m_{\bT^s}(g_k)$-norm, a fixed
function has to be
bounded in this norm as $k \to \infty$ (the constant $C$ in the last two inequalities 
does \emph{not} depend on $g$). Therefore,
we can deduce the $C^m_{\bT^s}(g_k)$-result from Lemma \ref{lem:tensorpull}.

In order to establish the $L^2_{m,\bT^s}$-result, we observe that the $g'_k$-volume
is $k^{\dim M}$ 
times a fixed volume form. 
So, over a ball $B_{{\rm flat}} \subset M$ where Proposition
\ref{adiabThm} holds with $\varepsilon=1/2$, the $g_k$-volume is $\mathcal{O}(k^{\dim
M})$
times a fixed volume form. Hence, with respect to $g_k$, the volume of $\PE|_{B_{{\rm flat}}}$
is $\mathcal{O}(k^{\dim M})$. Cover $M$ with finitely many
such balls, $B_{{\rm flat},i}$. Then, the volume $\vol_k$ of $\PE$, with respect to $g_k$,
satisfies
\[\vol_k \leq \sum_i \vol\left(\PE|_{B_{{\rm flat},i}}\right)=\mathcal{O}(k^{\dim
M}).\]
With all that in our hands, the result follows from the $C^m_{\bT^s}$-result and
the fact that $\| \phi \|_{L^2_{m,\bT^s}(g_k)} \leq \vol_k^{1/2} \| \phi \|_{C^m_{\bT^s}(g_k)}$.
\end{proof}

Now, we have everything we need in order to transfer the ``product results''
from Section~\ref{subsec:AnaLocModel} to $(\PE,J,\omega_{k,n})$. The next lemma is exactly the same as \cite[Lemma~5.8]{F}, thus we shall omit its proof since restricting to the $\bT^s$-invariant functions in the respective Sobolev spaces doesn't change it (cf. Remark \ref{rem:TsInvarFunc}).
\begin{lemma}[cf. Lemma 5.8 in \cite{F}]\label{lem:globSobolev}
For indices $m,l$ and $q \geq p$ satisfying $m -\dim_{\mathbb{R}} (\PE)/p
\geq l- \dim_{\mathbb{R}} (\PE)/q$, there is a constant $c$ (depending \emph{only}
on $m,l,q$ and $p$, but \emph{not} on $k$) such that for all $\phi \in L^p_{m,\bT^s} \left( \PE \right)$
and all sufficiently large $k$,
\[ 
\| \phi \|_{L^q_{l,\bT^s} (g_k)} \leq c \| \phi \|_{L^p_{m,\bT^s} (g_k)}.
\]
For indices $p,m$ satisfying $m-\dim_{\mathbb{R}} (\PE)/p>0$, there exists
a constant $c$ (depending \emph{only} on $p,m$ and \emph{not} on $k$), such that for all $\phi
\in L^p_{m,\bT^s} \left( \PE \right)$ and all sufficiently large $k$,
\[ \| \phi \|_{C^0_{\bT^s} (g_k)} \leq c \| \phi \|_{L^p_{m,\bT^s} (g_k)}.\]
\end{lemma}

We are now in a position to prove a uniform elliptic 
estimate for $L_{{\rm AEMO},\omega_{k,n}}^{{\bf 0}, 0}$ (N.B. $(\Theta,\oR)=0$ here).
\begin{lemma}[cf. Lemma 5.9 in \cite{F}]\label{lem:UnifEllipEst}
There exists a constant $C$, depending only on $m$, such that for all $\phi \in
L^2_{m+4,\bT^s} (\PE)$ and all sufficiently large $k$,
\[
\| \phi \|_{L^2_{m+4,\bT^s}(g_k)} \leq C \left( \| \phi  \|_{L^2_{\bT^s} (g_k)}+\|L_{{\rm AEMO},\omega_{k,n}}^{{\bf 0}, 0} (\phi) \|_{L^2_{m,\bT^s} (g_k)} \right),
\]
where as in Lemma \ref{lem:ParamLin} above, $L_{{\rm AEMO},\omega_{k,n}}^{{\bf 0}, 0}$ 
is the linearisation, for $(\Theta,\oR)=0$, of the ``parametrised extremal metric operator''
on ($\bT^s$-invariant) K\"ahler potentials determined by $\omega_{k,n}$.
\end{lemma}
\begin{proof}
Even though the elliptic operators under consideration are different, the proof is similar to the one of \cite[Lemma~5.9]{F}.

Following the strategy of proof of \cite[Lemma~5.9]{F}, one makes two observations:
\begin{itemize}
\item Applying Lemma \ref{lem:HamilVar} to the parts of
$L_{{\rm AEMO},\omega_{k,n}}^{{\bf 0}, 0}, L_{{\rm AEMO},\omega'_{k}}$ corresponding to the linearisations of $Q(\omega_{k,n}+ \idd \phi, B), \ Q(\omega'_k + \idd \phi, B)$ shows that---since both Hamiltonians are formed with respect to the same vector field $B$ and varied by the same invariant K\"ahler potential---the difference of their variations (linearisations) is \emph{zero} by using the first equality in equation \eqref{eq:HamilVar} (recall that $J'|_{B_{{\rm flat}}}=J|_{B_{{\rm flat}}}$, so we don't have to worry about $J$ in the first equality of equation \eqref{eq:HamilVar}).
\item For the parts of 
$L_{{\rm AEMO},\omega_{k,n}}^{{\bf 0}, 0}, L_{{\rm AEMO},\omega'_{k}}$ corresponding to the linearisations of the scalar curvature maps $Scal(\omega_{k,n}+\idd \phi), \ Scal(\omega'_k +\idd \phi)$ on invariant K\"ahler potentials $\phi$, one can argue exactly as in the proof of \cite[Lemma~5.9]{F}.
\end{itemize}
These two observations enable us to replace $L_{{\rm AEMO},\omega'_{k}}$ with $L_{{\rm AEMO},\omega_{k,n}}^{{\bf 0}, 0}$ in Lemma \ref{EllipEst}, just as in the case treated in \cite[Lemma~5.9]{F}, and hence we conclude.
\end{proof}

\subsection{Global Analysis}\label{GA}

In this section we will derive the global estimates, in order to find a lower
bound for the first non-zero eigenvalue of the operator $L_{{\rm AEMO},\omega_{k,n}}^{\Theta, \oR}$.
Following \cite[Section~6]{F} we will construct a \emph{global
model}, which has the crucial property of being a \emph{Riemannian submersion}
for $\PE \to (M, k \omega_M)$. 

The current section is similar in nature to \cite[Section~6]{F}, and many of the results presented here are
a variation of Fine's results. In particular, the construction of the
\emph{global model} used below is due to Fine---our analysis is slightly
different however, since we work with an operator involving parameters and have to 
deal with a non-trivial co-kernel.

In fact, the parameters $\Theta,\oR$ will play a crucial role to obtain the results below. 
As main result of this section, we are going to prove:
\begin{theorem}\label{thm:InverseEst}
For all large $k$ and suitable $n$,
the operator $L_{{\rm AEMO},\omega_{k,n}}^{\Theta, \oR} : L^2_{m+4, \bT^s} \times \bR^{s+1} \to L^2_{m, \bT^s}$ is a Banach space epimorphism. There exist a constant
$C$ and parameters $(\Theta,\oR) \in \bR^{s+1}$, 
such that for all large $k$ and all functions $\phi \in L^2_{m,\bT^s}$, the
right-inverse operator $I_{{\rm AEMO},\omega_{k,n}}^{\Theta, \oR}$ satisfies the estimate
\begin{equation}\label{eq:InvEsteq}
\| I_{{\rm AEMO},\omega_{k,n}}^{\Theta, \oR} ( \phi ) \|_{L^2_{m+4,\bT^s} (g_{k})} \leq C k^3 \| \phi  \|_{L^2_{m,\bT^s} (g_{k})}.
\end{equation}
\end{theorem}
Proving such an estimate is a genuine global issue. Therefore we are now going to describe the \emph{global model}, first constructed in \cite[Section~6.1]{F}.

\subsubsection{The global model}

We define a Riemannian metric $h_k$ on $\PE$ by using the fibrewise metrics
determined by the purely vertical part of $i \Lcurv$ (for the definition of $i \Lcurv$, see Proposition \ref{FPProp}), i.e. $\omega_{FS}$,
 and adding the metric $k \omega_M$
(in horizontal directions). In this setup, $(\PE,h_k)
\to (M,k \omega_M)$ is a \emph{Riemannian
submersion}.

With this construction, $g_{k,0}=h_k+a$, for some purely horizontal tensor
$a \in s^2 (T^* \PE)$, independent of $k$ (it is given by the
horizontal components of $i \Lcurv$). Horizontal 1-forms scale by $k^{-1/2}$
in the metric $h_k$, so we have for $k$ sufficiently large
\begin{equation}\label{eq:normdist}
\| g_{k,0}-h_k  \|_{C^0 (h_k)} \leq \frac{1}{2}.
\end{equation}
Also since $\|g_k - g_{k,0} \|_{C^0 (h_k)}=\mathcal{O}(k^{-1})$,
the inequality \eqref{eq:normdist} holds with $g_{k,0}$ replaced by $g_k$.
From all this one infers that the difference in the induced metrics
on $T^* \PE$ is uniformly bounded and hence the $L^2$-norms
on tensors determined by $h_k$ and $g_k$ are uniformly equivalent (this
will be crucial in the sequel).
\begin{lemma}[cf. Lemma 6.2 in \cite{F}]\label{lem:GlobNormEqv}
Let $T \to \PE$ be any bundle of tensors. Then there exist positive
constants $s,S$, such that $\forall t \in \Gamma (T)$ and sufficiently large
$k$ we have the equivalence of norms
\[s\|t\|_{L^2 (h_k)} \leq \|t\|_{L^2 (g_k)} \leq S \|t\|_{L^2 (h_k)}.\]
\end{lemma}

\subsubsection{Controlling the lowest eigenvalue of the parametrised Lichnerowicz operator}

As shown in Lemma \ref{lem:ParamLin}, we have $L_{{\rm AEMO},\omega_{k,n}}^{\Theta, \oR}(\phi) = \LichOp (\phi) - Q(\omega_{k,n},B_{\Theta}) - \oR + \O (k^{-n-2})$; with $\dneu= \dbar \nabla$, where $\dbar$
is the $\dbar$-operator on the holomorphic tangent bundle of $\PE$, $\nabla$ the gradient, and $\dneu^*$
is the $L^2$-adjoint of $\dneu$.
$\LichOp$ depends on the metric corresponding to $\omega_{k,n}$ and hence also on $k$. 
Since it is notationally more convenient, we shall just write $\nabla$ for $\nabla_{g_k}$,
$\dbar$ for $\dbar_{g_k}$ and $\dneu$ for $\dneu_{g_k}$.

The bound for the lowest non-zero eigenvalue of the ``parametrised Lichnerowicz operator'' $\LichOp (\phi) - Q(\omega_{k,n},B_{\Theta}) - \oR$ will be found by linking
together two eigenvalue estimates: the first being the one for the
ordinary Hodge Laplacian (Lemma \ref{lem:HodgeEst}), and the second being the one for the $\dbar$-Laplacian
acting on sections of the holomorphic tangent bundle (Lemma \ref{lem:dbarEst}).
\begin{lemma}[cf. Lemma 6.5 in \cite{F}]\label{lem:HodgeEst}
There exists a constant $C_1>0$ such that for all functions $\phi$
with $g_k$-mean value zero and all sufficiently large $k$,
\begin{equation}\label{eq:HodgeEst}
\|d \phi\|^2_{L^2(g_k)} \geq C_1 k^{-1} \|\phi\|^2_{L^2(g_k)}.
\end{equation}
\end{lemma}
\begin{proof}
The proof of this Lemma is, up to dimension considerations, the same as the proof of \cite[Lemma~6.5]{F}; however, for the reader's convenience we will provide the details.
One can find a constant $w$ such that $\phi-w$ has $h_1$-mean value zero.
Since $d \phi=d (\phi-w)$, using Lemma \ref{lem:GlobNormEqv} gives
$\|d \phi\|_{L^2(g_k)} \geq {\rm const} \|d(\phi-w)\|_{L^2(h_k)}$.
Let $|\cdot|_{h_k}$ denote the norm induced by the pointwise inner product
defined by $h_k$. By definition of $h_k$, we have $|d(\phi-w)|^2_{h_k}
\geq k^{-1} |d(\phi-w)|^2_{h_1}$; and since the volume form satisfies
$d \vol (h_k) \geq k^{\dim M} d \vol (h_1)$ we get
\[\|d(\phi-w)\|^2_{L^2(h_k)} \geq k^{(\dim M)-1} \|d(\phi-w)\|^2_{L^2(h_1)}.\]
Let $\mu_1$ be the first (non-zero) eigenvalue of the $h_1$-Laplacian. Using that
$\phi-w$ has mean value zero with respect to $h_1$ gives
\[\|d(\phi-w)\|^2_{L^2(h_1)} \geq \mu_1 \| \phi-w \|^2_{L^2(h_1)} \geq
\mu_1 k^{-\dim M} \| \phi-w \|^2_{L^2(h_k)}.\]
A further application of Lemma \ref{lem:GlobNormEqv} renders
\[\| \phi-w \|^2_{L^2(h_k)} \geq {\rm const} \| \phi-w \|^2_{L^2(g_k)}
\geq {\rm const} \| \phi \|^2_{L^2(g_k)},\]
whereas the second inequality follows from the assumption that $\phi$ has
$g_k$-mean value zero. Putting the inequalities together completes the proof.
\end{proof}

\begin{lemma}[cf. Lemma 6.6 in \cite{F}]\label{lem:dbarEst}
There exists a positive constant $C_2$ such that for all $\zeta= \nabla f$,
with $\zeta \perp \ker \dbar$, and sufficiently large $k$ we have
\begin{equation}\label{eq:dbarEst}
\|\dbar \zeta\|^2_{L^2(g_k)} \geq C_2 k^{-2} \| \zeta \|^2_{L^2(g_k)}.
\end{equation}
\end{lemma}
\begin{proof}
The proof is the same as the proof of \cite[Lemma~6.6]{F}, modified for our purposes as the proof of Lemma \ref{lem:HodgeEst} above.
In fact, up to dimension considerations, the proof is the same as in Fine's case since we assume $\zeta \perp \ker \dbar$.
\end{proof}
Linking the two estimates just proved gives us an estimate for $\dneu$.
\begin{lemma}[cf. Lemma 6.7 in \cite{F}]\label{lem:dneuEst}
There exists a constant $C$ such that for all $\phi \perp \ker \dneu$
and sufficiently large $k$,
\begin{equation}\label{eq:dneuEst}
\|\dneu \phi \|^2_{L^2(g_k)} \geq C k^{-3} \|\phi\|^2_{L^2(g_k)}.
\end{equation}
\end{lemma}
\begin{proof}
The same proof as in \cite[Lemma~6.7]{F} works here as well:
Combining Lemmas \ref{lem:HodgeEst} and \ref{lem:dbarEst} shows that when
$\phi \perp \ker \dneu$,
\[\|\dbar \nabla \phi\|^2_{L^2(g_k)} \geq C_2 k^{-2} \|\nabla \phi\|^2_{L^2(g_k)}
= C_2 k^{-2} \| d \phi\|^2_{L^2(g_k)} \geq C_1 C_2 k^{-3} \|\phi\|^2_{L^2(g_k)}.\]
\end{proof}
From this Lemma, it follows 
that for $\phi \perp \ker \LichOp$,
\begin{equation}\label{eq:CSEVEst}
\|\LichOp \phi \|_{L^2(g_k)} \geq C k^{-3} \|\phi\|_{L^2(g_k)}.
\end{equation}
\begin{remark}\label{rem:cokerLichOp}
The elements $f \in \ker \LichOp \cong \coker \LichOp$ can be identified with the (real holomorphic) Hamiltonian Killing vector fields on the underlying (compact) K\"ahler manifold via the Hamiltonian construction, cf. Remark \ref{rem:LichOperator}. In our situation all Hamiltonian Killing vector fields on $\PE$ are induced by the bundle endomorphisms $Id_{E_1}, \dots, Id_{E_s}$ as in Definition \ref{def:InfAction}, since the bundle $E$ splits as a direct sum of stable subbundles all having different slope and the base admits no holomorphic automorphisms. Therefore, the parameters $(\Theta,\oR) \in \bR^{s+1}$ can be chosen such that the projection $\proj_{\ker \LichOp} \phi$ of any $\phi$ to $\ker \LichOp \cong \coker \LichOp$ can be written as $\proj_{\ker \LichOp} \phi=-Q(\omega_{k,n},B_{\Theta}) - \oR$. 
\end{remark}
Thus, the estimate \eqref{eq:CSEVEst} can be extended, for suitably chosen $(\Theta,\oR) \in \bR^{s+1}$, to all $\phi$ as
\[
\|\LichOp \phi - Q(\omega_{k,n},B_{\Theta}) - \oR\|_{L^2(g_k)} \geq C k^{-3} \|\phi\|_{L^2(g_k)}.
\]
We formulate this observation as a Lemma.
\begin{lemma}\label{lem:CSEVEstParam}
There exist a constant $C$ and parameters $(\Theta,\oR) \in \bR^{s+1}$ such that for all $\phi$ and sufficiently large $k$,
\begin{equation}\label{eq:CSEVEstParam}
\|\LichOp \phi - Q(\omega_{k,n},B_{\Theta}) - \oR\|_{L^2(g_k)} \geq C k^{-3} \|\phi\|_{L^2(g_k)}.
\end{equation}
\end{lemma}
\begin{remark}
Lemmas \ref{lem:HodgeEst}--\ref{lem:CSEVEstParam} were proved for functions $\phi$ not necessarily invariant under the $\bT^s$-action induced by $Id_{E_1}, \dots, Id_{E_s}$ on $\PE$. However, restricting to $\bT^s$-invariant functions does not affect the proofs and the results are valid for such functions as well
(cf. also Remark \ref{rem:TsInvarFunc}).
\end{remark}

\subsubsection{Controlling the (right-)inverse}

\begin{lemma}[cf. Lemma 6.8 in \cite{F}]\label{lem:UnifInv}
There is a constant $C$, depending only on $m$, and parameters $(\Theta,\oR) \in \bR^{s+1}$, such that for all $\phi \in
L^2_{m+4,\bT^s}$ and sufficiently large $k$,
\[\| \phi \|_{L^2_{m+4,\bT^s}(g_k)} \leq C \left( \| \phi  \|_{L^2_{\bT^s}(g_k)}+\|\LichOp
(\phi) - Q(\omega_{k,n},B_{\Theta}) - \oR\|_{L^2_{m,\bT^s} (g_k)} \right).\]
\end{lemma}
\begin{proof}
The proof is very similar to the one in \cite[Lemma~6.8]{F}.
Using Lemma \ref{lem:ParamLin} with $(\Theta,\oR)=({\bf 0},0)$, we have
\[L_{{\rm AEMO},\omega_{k,n}}^{{\bf 0}, 0}(\phi) = \LichOp (\phi) + \O (k^{-n-2}).\]
Since by equation \eqref{eq:LAEMOLichOpPara} and Lemma \ref{lem:UnifConv}, the $\O (k^{-n-2})$-terms tend to zero in the $C^m_{\bT^s} (g_k)$-norm, $L_{{\rm AEMO},\omega_{k,n}}^{{\bf 0}, 0} - \LichOp$ converges to zero in the operator norm induced by the $L^2_{m,\bT^s} (g_k)$-Sobolev norm. Hence the estimate follows for $(\Theta,\oR)=({\bf 0},0)$ from Lemma \ref{lem:UnifEllipEst}. Choosing the parameters $(\Theta,\oR)$ as in Remark \ref{rem:cokerLichOp} we obtain the desired estimate.
\end{proof}

Now, everything is in place to prove
\begin{theorem}\label{thm:LichOpInvEst}
The operator $\LichOp - Q(\omega_{k,n},B_{\Theta}) - \oR: L^2_{m+4,\bT^s} \times \bR^{s+1} \to
L^2_{m,\bT^s}$ is a Banach space epimorphism. There exist
a constant $S$, depending only on $m$, and parameters $(\Theta,\oR)$, such that for all large $k$ and all $\rho \in L^2_{m,\bT^s}$, 
the right-inverse operator $W_{\omega_{k,n}}^{\Theta,\oR}$ satisfies
\[
\|W_{\omega_{k,n}}^{\Theta,\oR} \rho\|_{L^2_{m+4,\bT^s}(g_k)} \leq S k^3 \| \rho \|_{L^2_{m,\bT^s} (g_k)}.
\]
\end{theorem}
\begin{proof}
Since $\LichOp$ is a fourth-order, linear-elliptic and self-adjoint differential operator,
the right-inverse $W_{\omega_{k,n}}^{\Theta,\oR}$ of $\LichOp - Q(\omega_{k,n},B_{\Theta}) - \oR$ exists since we can vary the parameters $(\Theta,\oR) \in \bR^{s+1}$ such that we can deal with the (co-)kernel of $\LichOp$ (see Remark \ref{rem:cokerLichOp}).
It follows from Lemma \ref{lem:CSEVEstParam} applied to $\phi=W_{\omega_{k,n}}^{\Theta,\oR} \rho$, with the parameters $(\Theta,\oR)$ chosen such that they kill the projection of $\rho$ to $\coker \LichOp$, that there is a constant $C$ such that for all $\rho \in L^2_{m,\bT^s}$ we get
\[
\|W_{\omega_{k,n}}^{\Theta,\oR} \rho\|_{L^2_{\bT^s}(g_k)} \leq C k^3 \| \rho \|_{L^2_{\bT^s}(g_k)}. 
\]
By applying Lemma \ref{lem:UnifInv} to $\phi=W_{\omega_{k,n}}^{\Theta,\oR} \rho$, we obtain the required
bound.
\end{proof}
The following standard lemma, the proof of which shall be omitted, essentially states the openness of (right) invertibility in the Banach space of bounded linear operators endowed with a suitable operator norm.
\begin{lemma}\label{lem:FuncAnaLem}
Let $L,D: B_1 \to B_2$ be linear maps between Banach spaces.
If $D$ is a bounded right-invertible linear map with bounded right-inverse
$W$, such
that
\[\|L-D\| \leq (2 \|W\| )^{-1},\]
then $L$ is also right-invertible and has a bounded right-inverse $I$ satisfying
$\|I\|
\leq 2 \|W\|$.
\end{lemma}
Now we have all the ingredients for completing the proof of Theorem \ref{thm:InverseEst}.
\begin{proof}[Proof of Theorem \ref{thm:InverseEst}.]
By Lemma \ref{lem:ParamLin},
\[
L_{{\rm AEMO},\omega_{k,n}}^{\Theta, \oR}(\phi) = \LichOp (\phi) - Q(\omega_{k,n},B_{\Theta}) - \oR  
+ \O (k^{-n-2}),
\]
so by Lemma \ref{lem:UnifConv} there exists a constant $c$ such that in the
operator norm determined by the $g_k$-Sobolev norms, we have 
$\| L_{{\rm AEMO},\omega_{k,n}}^{\Theta, \oR} - ( \LichOp - Q(\omega_{k,n},B_{\Theta}) - \oR ) \| \leq c k^{-n-2+(\dim M)/2}$. 
Therefore, if $n$ and $k$ are sufficiently
large: 
\[
\|L_{{\rm AEMO},\omega_{k,n}}^{\Theta, \oR} - ( \LichOp - Q(\omega_{k,n},B_{\Theta}) - \oR ) \| \leq (2 \|W_{\omega_{k,n}}^{\Theta,\oR}\|)^{-1}.
\]
Now, Lemma \ref{lem:FuncAnaLem} shows that $L_{{\rm AEMO},\omega_{k,n}}^{\Theta, \oR}$ is 
right-invertible and provides us with a bound for its right-inverse
\[
\| I_{{\rm AEMO},\omega_{k,n}}^{\Theta, \oR} \| \leq 2 \| W_{\omega_{k,n}}^{\Theta,\oR} \| \leq C k^3,
\]
for some constant $C$.
\end{proof}

\subsection{Estimating the non-linear terms}\label{sec:nonlinTerms}

What remains in our discussion of the analysis is the issue of estimating
the \emph{non-linear terms} of the ``parametrised extremal metric operator'' 
\[
\AEMO^{\Theta, \oR} (\phi) := Scal(\omega_{k,n}+\idd \phi) - Q(\omega_{k,n}+ \idd \phi, B) - Q(\omega_{k,n},B_{\Theta}) - \oC - \oR,
\]
defined in \eqref{eq:ParaExtrMetOPDef}. This can be done in our case in a similar way as in 
\cite[Lemma~7.1]{F}. 

The operator corresponding to the non-linear terms of $\AEMO^{\Theta, \oR}$ shall be denoted by
\[
N^{\Theta,\oR}_k(\phi):=\AEMO^{\Theta, \oR}(\phi) - L_{{\rm AEMO},\omega_{k,n}}^{\Theta, \oR}(\phi),
\]
where the two operators on the right hand side are evaluated on the same $\bT^s$-invariant K\"ahler potential $\phi$.

\begin{proposition}\label{thm:nonlinEst}
Let $m > \dim_{\mathbb{C}} \PE$. There exist positive constants $c,K$ such
that for all
$\phi,\psi \in L^2_{m+4,\bT^s}(g_k)$ with $\|\phi\|_{L^2_{m+4,\bT^s}(g_k)},\|\psi\|_{L^2_{m+4,\bT^s}(g_k)}
\leq c$ and $k$ sufficiently large,
\begin{equation}\label{eq:nonlinEst}
\| N^{\Theta,\oR}_k (\phi) - N^{\Theta,\oR}_k (\psi) \|_{L^2_{m,\bT^s}(g_k)} \leq K \max \left\{ \|\phi\|_{L^2_{m+4,\bT^s}(g_k)},\|\psi\|_{L^2_{m+4,\bT^s}(g_k)} \right\} \| \phi - \psi \|_{L^2_{m+4,\bT^s}(g_k)}.
\end{equation}
\end{proposition}
\begin{proof}
The proof is similar to the one given in \cite[Lemma~7.1]{F}.
Using the mean value theorem gives
\[ \| N^{\Theta,\oR}_k (\phi) - N^{\Theta,\oR}_k (\psi) \|_{L^2_{m,\bT^s} (g_k)} \leq \sup_{\vartheta \in
[\phi,\psi]} \|(D N^{\Theta,\oR}_k)_{\vartheta} \| \| \phi - \psi \|_{L^2_{m+4,\bT^s} 
(g_k)}, \]
with $\| (D N^{\Theta,\oR}_k)_{\vartheta} \|$ being the operator norm of the derivative
of $N^{\Theta,\oR}_k$ at $\vartheta$; and
\[\vartheta \in [\phi,\psi] := \{ \vartheta \in L^2_{m+4,\bT^s} \ \text{such that}
\ \vartheta = \phi + t (\psi - \phi), \ \text{for some} \ t \in [0,1]  \}.\]

So $D N^{\Theta,\oR}_k = L_{{\rm AEMO},\omega_{k,n}+ \idd \vartheta}^{\Theta, \oR} - L_{{\rm AEMO},\omega_{k,n}}^{\Theta, \oR}$; where
$L_{{\rm AEMO},\omega_{k,n}+ \idd \vartheta}^{\Theta, \oR}$ is the linearisation
of the ``parametrised extremal metric operator'', defined in \eqref{eq:ParaExtrMetOPDef}, 
at $\omega_{k,n}+i \dbar \partial \vartheta$. We apply
\begin{itemize}
\item \cite[Lemma~2.10]{F}\footnote[1]{Since by our assumption $m > \dim_{\mathbb{C}} \PE$, the condition on the
indices in \cite[Lemma~2.10]{F} is fulfilled (\cite[Lemma~2.10]{F} holds for $\bT^s$-invariant functions, as well). This lemma also
requires the constants in the $g_k$-Sobolev inequalities to be uniformly
bounded---which was proven in Lemma \ref{lem:globSobolev}.
Moreover, it is required in order to apply \cite[Lemma~2.10]{F},
that the $C^m (g_k)$-norm of the curvature of $\omega_{k,n}$ is bounded
above---which follows from Proposition~\ref{adiabThm} and \cite[Lemma~2.7]{F}.} to the parts of 
$L_{{\rm AEMO},\omega_{k,n}+ \idd \vartheta}^{\Theta, \oR}, \ L_{{\rm AEMO},\omega_{k,n}}^{\Theta, \oR}$ corresponding to the linearisations of the scalar curvature maps $Scal(\omega_{k,n}+\idd \vartheta+\idd \nu), \ Scal(\omega_{k,n}+\idd \nu)$ on invariant K\"ahler potentials $\nu$, 

\item Lemma \ref{lem:HamilVar} to the parts of
$L_{{\rm AEMO},\omega_{k,n}+ \idd \vartheta}^{\Theta, \oR}, \ L_{{\rm AEMO},\omega_{k,n}}^{\Theta, \oR}$ corresponding to the linearisations of $Q(\omega_{k,n}+ \idd \vartheta+\idd \nu, B), \ Q(\omega_{k,n}+ \idd \nu, B)$---since both Hamiltonians are formed with respect to the same vector field $B$ and varied by the same invariant K\"ahler potential, the difference of their variations (linearisations) is \emph{zero} by using the first equality in equation \eqref{eq:HamilVar},

\item Lemma \ref{lem:HamilVar} to the $Q(\omega_{k,n}+\idd \vartheta,B_{\Theta}) - Q(\omega_{k,n},B_{\Theta})$-part of
$D N^{\Theta,\oR}_k = L_{{\rm AEMO},\omega_{k,n}+ \idd \vartheta}^{\Theta, \oR} - L_{{\rm AEMO},\omega_{k,n}}^{\Theta, \oR}$ (N.B. this parameter-dependent part of the operator is linear in the parameter $\Theta$, and \emph{not} linearised with respect to the invariant K\"ahler potential). We can estimate using the first equality in equation \eqref{eq:HamilVar}, 
\[
\| Q(\omega_{k,n}+\idd \vartheta,B_{\Theta}) - Q(\omega_{k,n},B_{\Theta}) \|_{L^2_{m,\bT^s}} \leq
C \| d \vartheta (J B_{\Theta}) \|_{L^2_{m,\bT^s}} 
\]
\begin{equation}\label{eq:LeibnizEst}
\leq C' \| J B_{\Theta} \|_{L^2_{m,\bT^s}} \| d \vartheta 
\|_{L^2_{m,\bT^s}} \leq C'' \| \vartheta \|_{L^2_{m+4,\bT^s}},
\end{equation} 
where $C,C',C''$ are constants. In the third inequality of \eqref{eq:LeibnizEst} we used the following inequality on tensors, derived from the Leibniz rule and further explained in \cite[Section~2.2.2]{F}.
For tensors $T,T' \in L^p_m$, we have
\begin{equation}\label{eq:TensorIneq}
\| T \star T'\|_{L^p_m} \leq C \| T \|_{L^p_m} \| T' \|_{L^p_m},
\end{equation}
where ``$\star$'' stands for any algebraic operation
consisting of tensor products and contractions. Here,
the constant $C$ depends only on $m$, and not on the metric
determining the norm; this follows from the uniform bound on the constants in the $g_k$-Sobolev inequalities
(see \cite[Section~2.2.2]{F} for details). 
\end{itemize}
Putting the three points above together gives us the estimate
\[
\| L_{{\rm AEMO},\omega_{k,n}+ \idd \vartheta}^{\Theta, \oR} - L_{{\rm AEMO},\omega_{k,n}}^{\Theta, \oR} \| \leq {\rm const}
\| \vartheta \|_{L^2_{m+4,\bT^s} (g_k)}.
\]
Since for all $\vartheta \in [\phi,\psi]$,
\[ \| \vartheta \|_{L^2_{m+4,\bT^s}} \leq \max \left\{ \|\phi\|_{L^2_{m+4,\bT^s}},
\|\psi\|_{L^2_{m+4,\bT^s}} \right\}, \]
the result follows.
\end{proof}

\subsection{Applying the implicit function theorem}\label{ApplIFT}

In this section we will complete the proof of our main result, Theorem \ref{Thm:MainResult}, 
by using the parameter-dependent Implicit Function Theorem (Theorem \ref{IFT}). 

\begin{proof}[Proof of Theorem \ref{Thm:MainResult}]
For all $k \gg 0$ and sufficiently large $n$, the ``parametrised extremal metric operator''
\[
\AEMO^{\Theta, \oR} : L^2_{m+4,\bT^s} \times \mathbb{R}^{s+1} \to
L^2_{m,\bT^s} 
\]
satisfies
\begin{itemize}
\item $\AEMO^{{\bf 0}, 0} (0) = \mathcal{O}(k^{-n-2+(\dim
M)/2})$ in $L^2_{m,\bT^s} (g_k)$, by Lemma \ref{lem:UnifConv}.
\item Its linearisation at $\omega_{k,n}$, $L_{{\rm AEMO},\omega_{k,n}}^{\Theta, \oR} : L^2_{m+4, \bT^s} \times \bR^{s+1} \to L^2_{m, \bT^s}$, is a Banach space epimorphism
with right-inverse $I_{{\rm AEMO},\omega_{k,n}}^{\Theta, \oR}$, which is $\mathcal{O}(k^3)$ in operator norm by Theorem \ref{thm:InverseEst}.
\item There exists a constant $K$ such that for all sufficiently small $V$,
the non-linear piece $N^{\Theta,\oR}_k$ of $\AEMO^{\Theta, \oR}$ is Lipschitz with
constant $V$ on a ball about $0$ of radius $KV$. This follows from Proposition \ref{thm:nonlinEst}.
\item There is only \emph{one} $\bT^s$-action on $\PE$, generated by $Id_{E_1},\dots,Id_{E_s} \in \End (E)$.
This allows us to deal with the non-trivial co-kernel of $\LichOp$ by varying the parameters $(\Theta,\oR)
\in \bR^{s+1}$, see Remark \ref{rem:cokerLichOp}.
In the end, there is only \emph{one} choice for the parameters $(\Theta,\oR)$,
since $\oC+\oR$ in equation \eqref{eq:ExtrEqKaehlerPotPara} is a topological constant (the average scalar curvature), and the parameters $\Theta$ are determined by the Futaki invariant which by Definition \ref{def:EVF} is dual---with respect to the Futaki-Mabuchi inner product---to the extremal vector field.
\end{itemize}
By the Implicit Function Theorem (Theorem \ref{IFT}), the second and third points above
imply that the radius $\delta'_k$ of
the ball about the origin on which $N^{\Theta,\oR}_k$ is Lipschitz with constant
$(2 \| I_{{\rm AEMO},\omega_{k,n}}^{\Theta, \oR} \|)^{-1}$, is bounded below by $C k^{-3}$ for some 
constant $C>0$. Since $\delta_k=\delta'_k (2 \| I_{{\rm AEMO},\omega_{k,n}}^{\Theta, \oR} \|)^{-1}$, it follows that $\delta_k$
is bounded by $C k^{-6}$ for some constant $C>0$.

Looking at Theorem \ref{IFT}, we see that for $\varrho
\in L^2_{m,\bT^s}$ with $\|\AEMO^{{\bf 0}, 0} (0) - \varrho \|_{L^2_{m,\bT^s}
(g_k)} \leq C k^{-6}$,
the equation $\AEMO^{\Theta, \oR} (\phi) =\varrho$ has a solution.
The first of the above properties implies then, that for sufficiently large $n$ and
$k \gg 0$, the equation $\AEMO^{\Theta, \oR} (\phi)=0$ has
a solution $(\phi,\Theta,\oR)$ with $\phi \in L^2_{m+4,\bT^s}
(g_k)$, 
where the parameters $\Theta$ and $\oS$ are determined
as in the fourth point above.

Provided $m$ is big enough such that $L^2_{m+4,\bT^s} \hookrightarrow C^{2,\alpha}_{\bT^s}$,
applying the regularity Lemma \ref{lem:MainReg} from below iteratively shows that $\phi$
is smooth.

\end{proof}

In order to carry out our arguments above, we still need to establish a \emph{regularity result} about
extremal K\"ahler metrics. This will ensure that the $\bT^s$-invariant K\"ahler potential $\phi$, found in Section \ref{ApplIFT}, is smooth.

As already mentioned in equation \eqref{eq:ELCalabieq}, a K\"ahler metric $g$ on a (compact) K\"ahler manifold $(M,J,g,\omega)$ is \emph{extremal}
if the gradient of its scalar curvature
$\nabla_g Scal(g)$ preserves the complex structure $J$, i.e. it is the real part of a holomorphic section
of $T^{1,0} M$. So, instead of using equation~\eqref{eq:ExtrCond1}, another condition for a K\"ahler metric to be extremal is
\begin{equation}\label{CalELeq}
\mathcal{L}_{\nabla_g Scal(g)} J=0,
\end{equation}
where $\mathcal{L}$ denotes the Lie-derivative.

The extremal K\"ahler metric we constructed in Theorem \ref{Thm:MainResult} therefore satisfies Equation \eqref{CalELeq}, and we will use this equation to prove the 
following regularity result (similar results were already proven in \cite[Lemma~2.3]{F} and 
\cite[Proposition~4]{LS1}).
\begin{lemma}\label{lem:MainReg}
If the K\"ahler metric $g_{\phi}$ corresponding to $\omega_{\phi}=\omega+i \dbar \partial
\phi$, on a compact K\"ahler manifold,
is \emph{extremal} with $\phi \in C^{m,\alpha}$, $m \geq 2$,
then $\phi \in C^{m+3,\alpha}$.
\end{lemma}
\begin{proof} 
We follow the proof of \cite[Lemma~2.3]{F}.
For an extremal K\"ahler metric $g$,
the gradient of the scalar curvature $\nabla_g Scal(g)$ is the real part of a holomorphic vector field, hence it is \emph{real-analytic}.
It therefore follows that the metric dual of $\nabla_{g_{\phi}} Scal(g_{\phi})$, i.e. $d Scal (g_{\phi})$ is of class $C^{m-2,\alpha}$ (as the metric $g_{\phi}$ corresponding to $\omega_{\phi}$
is of class $C^{m-2,\alpha}$); so $Scal(g_{\phi})$ 
is therefore of class $C^{m-1,\alpha}$.

Now, $Scal(g_{\phi})= \Delta_{g_{\phi}} U$, where $\Delta_{g_{\phi}}$ is the $g_{\phi}$-Laplacian
and
\[U = - \log \det (g + \Phi),\]
where $\Phi$ is the real symmetric tensor corresponding to the (1,1)-form
$i \dbar \partial \phi$, and $g$ is the K\"ahler metric corresponding to $\omega$.

Since $\phi \in C^{m,\alpha}$, $\Delta_{g_{\phi}}$ is a linear second order elliptic
operator with coefficients in $C^{m-2,\alpha}$. By standard elliptic regularity
results (cf. \cite[Theorem~3.59]{Aub}) and since $Scal(g_{\phi}) \in C^{m-1,\alpha}$,
we get $U \in C^{m,\alpha}$.

The map $\phi \mapsto - \log \det (g+\Phi)$ is non-linear, but also second
order and elliptic. Therefore, it also satisfies an elliptic 
regularity result 
(cf. \cite[Theorem~3.56]{Aub}), hence $\phi \in C^{m+2,\alpha}$.

Therefore, $\Delta_{g_{\phi}}$ has $C^{m,\alpha}$-coefficients;
hence $U \in C^{m+1,\alpha}$ and $\phi \in C^{m+3,\alpha}$.
\end{proof}



\

\noindent
\textsc{D\'epartement de Math\'ematiques, Universit\'e Libre de Bruxelles, CP 218, Boulevard du Triomphe, B-1050 Bruxelles, Belgium}\\
\textit{Email:} Till ``\emph{dot}'' Broennle ``\emph{at}'' ulb.ac.be

\end{document}